\documentclass[11pt]{amsart}
\usepackage{latexsym,amscd,amssymb, graphicx, shuffle, young}  
\usepackage[margin=1in]{geometry}

\newtheorem{theorem}{Theorem}[section]
\newtheorem{proposition}[theorem]{Proposition}
\newtheorem{corollary}[theorem]{Corollary}
\newtheorem{lemma}[theorem]{Lemma}

\newtheorem{observation}[theorem]{Observation}

\newtheorem{example}[theorem]{Example}

\newtheorem{defn}[theorem]{Definition}
\theoremstyle{definition}

\newcommand{\mult}{{\mathrm {mult}}}

\newcommand{\NC}{{\mathsf {NC}}}
\newcommand{\HNC}{{\mathsf {HNC}}}
\newcommand{\Krew}{{\mathsf {Krew}}}
\newcommand{\krew}{{\mathtt {krew}}}
\newcommand{\Park}{{\mathsf {Park}}}
\newcommand{\rot}{{\mathtt {rot}}}
\newcommand{\rfn}{{\mathtt {rfn}}}
\newcommand{\rank}{{\mathrm{rank}}}

\newcommand{\Cat}{{\mathsf{Cat}}}
\newcommand{\Nar}{{\mathsf{Nar}}}

\newcommand{\symm}{{\mathfrak{S}}}

\newcommand{\CC}{{\mathbb {C}}}
\newcommand{\ZZ}{{\mathbb {Z}}}

\begin{document}

\title[Cyclic sieving and rational Catalan theory]
{Cyclic sieving and rational Catalan theory}

\author{Michelle Bodnar and Brendon Rhoades}
\address
{Deptartment of Mathematics \newline \indent
University of California, San Diego \newline \indent
La Jolla, CA, 92093-0112, USA}
\email{(mbodnar, bprhoades)@math.ucsd.edu}

\begin{abstract}
Let  $a < b$ be coprime positive integers.  Armstrong, Rhoades, and Williams \cite{ARW} defined a set 
$\NC(a,b)$ of `rational noncrossing partitions', which form a subset of the ordinary noncrossing partitions
of $\{1, 2, \dots, b-1\}$.
Confirming a conjecture of Armstrong et. al., we prove that $\NC(a,b)$ is closed under rotation and prove
an instance of the cyclic sieving phenomenon for this rotational action.
We also define a rational generalization of the $\symm_a$-noncrossing parking functions of
Armstrong, Reiner, and Rhoades \cite{ARR}.
\end{abstract}

\maketitle

\section{Introduction}
\label{Introduction}

This paper is about generalized noncrossing partitions arising in rational Catalan combinatorics.
A set partition of $[n] := \{1, 2, \dots, n\}$ is noncrossing if its blocks do not cross when drawn on a disk 
whose boundary is
labeled 
clockwise with $1, 2, \dots, n$.  Noncrossing partitions play a key role in algebraic and geometric combinatorics.
Along with an ever-expanding family of other combinatorial objects,
the noncrossing partitions of $[n]$ are famously counted by the {\em Catalan number}
$\Cat(n) = \frac{1}{n+1}{2n \choose n} = \frac{1}{2n+1}{2n+1 \choose n}$.

Given a Fuss parameter $m \geq 1$, the {\em Fuss-Catalan number} is 
$\Cat^{(m)}(n) = \frac{1}{mn+n+1} {mn+n+1 \choose n}$.  When $m = 1$, this reduces to the classical Catalan
number. Many Catalan object have natural Fuss-Catalan generalizations.  In the case of noncrossing 
partitions, the Fuss-Catalan number $\Cat^{(m)}(n)$ counts set partitions of $[mn]$
which are {\em $m$-divisible} in the sense that every block has size divisible by $m$.

For coprime positive integers $a$ and $b$, the {\em rational Catalan number} is 
$\Cat(a,b) = \frac{1}{a+b} {a+b \choose a, b}$.  Observe that 
$\Cat(n, n+1) = \Cat(n)$ and $\Cat(n, mn+1) = \Cat^{(m)}(n)$, so that rational Catalan numbers
are a further generalization of Fuss-Catalan numbers.
Inspired by favorable representation theoretic properties of the rational Cherednik algebra 
attached to the symmetric group $\symm_a$ at parameter $\frac{b}{a}$, the research program
of {\em rational Catalan combinatorics} seeks to further generalize Catalan combinatorics
to the rational setting.

Some rational generalizations of Catalan objects have been around for decades -- the rational analog
of a Dyck path dates back at least to the probability literature of the 1940s.
Armstrong, Rhoades, and Williams used rational Dyck paths to define rational analogs of 
polygon triangulations, noncrossing perfect matchings, and noncrossing partitions \cite{ARW}.
This paper goes deeper into the study of rational noncrossing partitions.

For coprime parameters $a < b$, Armstrong et. al. defined the {\em $a,b$-noncrossing partitions}
to be a subset $\NC(a,b)$ of the collection of noncrossing partitions of $[b-1]$ arising 
from a laser construction involving rational Dyck paths (see Section~\ref{Background} for details).
It was shown that $\NC(a,b)$ is counted by $\Cat(a,b)$, as it should be, and that 
$\NC(n,mn+1)$ is the set of $m$-divisible noncrossing partitions of $[mn]$, as it should be.

However, the construction of $\NC(a,b)$ in \cite{ARW} was indirect and involved the intermediate 
object of rational Dyck paths.  This left open the question of whether many of the fundamental properties
of classical noncrossing partitions generalize to the rational case.  
For example, it was unknown whether the set $\NC(a,b)$ is closed under the dihedral group of symmetries
of the disk with $b-1$ labeled boundary points.
Consequently, the rich theory of counting noncrossing set partitions fixed by a dihedral symmetry
(see \cite{RSWCSP}) lacked a rational extension.
Moreover, the unknown status of rotational closure made
 it difficult to generalize the noncrossing parking functions of Armstrong, Reiner, and
Rhoades \cite{ARR} (or the $2$-noncrossing partitions of Edelman \cite{Edelman}) to the rational setting.
The core problem was that the natural dihedral symmetries of noncrossing partitions
are harder to visualize on the level of Dyck paths, even in the classical case.

The purpose of this paper is to resolve the issues in the last paragraph to support $\NC(a,b)$ as 
the `correct' definition of the rational noncrossing partitions.  We will prove the following. 
\begin{itemize}
\item
$\NC(a,b)$ is closed under
dihedral symmetries (Corollary~\ref{dihedral-closed}).
\item 
The action of rotation on $\NC(a,b)$ exhibits a cyclic sieving phenomenon
generalizing that for the action of rotation on classical noncrossing partitions (Theorem~\ref{catalan-csp}). 
\item
The numerology of partitions in $\NC(a,b)$ with a nontrivial rotational symmetry generalizes that of 
classical noncrossing partitions with a nontrivial rotational symmetry 
(Corollaries~\ref{symmetric-kreweras-count}, \ref{symmetric-narayana-count}, \ref{symmetric-catalan-count}).
\item
Partitions in $\NC(a,b)$ can be decorated to obtain a $\symm_a \times \ZZ_{b-1}$-set of 
{\em rational noncrossing parking functions} $\Park^{NC}(a,b)$.  The formula for the
permutation character of this set 
generalizes the corresponding formula for the classical case (Theorem~\ref{rational-weak}).
\end{itemize}

The key to obtaining the rational extensions of classical results presented above will be to 
develop a better understanding of the set $\NC(a,b)$.  
We will give two new characterizations of this set 
(Propositions~\ref{kreweras-characterization} and \ref{rank-sequence-characterization}).
The more important of these will involve an idea genuinely new to rational Catalan combinatorics:
a new measure of size for blocks of set partitions in $\NC(a,b)$ called {\em rank}.

In the Fuss-Catalan case $(a,b) = (n, mn+1)$, the rank of a block is determined by its 
cardinality (Proposition~\ref{ranks-fuss}).  However, rank and cardinality diverge at the rational level 
of generality.
The main heuristic  of this paper is that: 
\begin{quote}
The rank of a block  of a rational noncrossing partition
 is a better measure of its
size than its cardinality.
\end{quote}
In Section~\ref{Characterizations}, we will prove that rank shares additivity and dihedral invariance properties
with size.  
We will prove that ranks (unlike cardinalities) characterize which noncrossing set partitions of $[b-1]$
lie in $\NC(a,b)$.
In Section~\ref{Parking functions}, ranks (not cardinalities) will be used to define and study rational
noncrossing parking functions.
In Sections~\ref{Modified rank sequences} and \ref{Cyclic sieving}, ranks (not cardinalities)
will be used to give rational analogs of enumerative results for noncrossing partitions
with rotational symmetry.

\section{Background}
\label{Background}

\subsection{Rational Dyck paths}
The prototypical object in rational Catalan combinatorics is the rational   Dyck path.
Let $(a, b)$ be coprime positive integers.  An {\em $a,b$-Dyck path} (or just a {\em Dyck path} when $a$ and $b$ are clear
from context) is a lattice path in $\ZZ^2$ consisting of north and east steps which starts at $(0, 0)$ ends at $(b,a)$, and stays
above the line $y = \frac{a}{b} x$.  
The $5,8$-Dyck path $NENEENNENEEE$  is shown in Figure~\ref{fig:nc-example}.
When $(a,b) = (n,n+1)$, rational Dyck paths are equivalent to classical Dyck paths 
 -- lattice paths from $(0, 0)$ to $(n, n)$ which stay weakly above $y = x$.

If $D$ is an $a,b$-Dyck path, a {\em valley} of 
$D$ is a lattice point $P$ on $D$ such that $P$ is immediately preceded by an east
step and succeeded by a north step.  
A {\em vertical run} of $D$ is a maximal contiguous sequence of north steps; the number of vertical runs 
equal the number of valleys.  The $5,8$-Dyck path shown in Figure~\ref{fig:nc-example} has $4$ valleys.  The
vertical runs of this path have sizes $1, 1, 2,$ and $1$. 

The numerology associated to rational Dyck paths generalizes that of classical Dyck paths.
The number of $a,b$-Dyck paths is the rational Catalan number $\Cat(a,b) = \frac{1}{a+b}{a+b \choose a,b}$.
The set of $a,b$-Dyck paths $D$ with $k$ vertical runs is counted by the {\em rational Narayana number}
\begin{equation}
\Nar(a,b; k) := \frac{1}{a} {a \choose k} {b-1 \choose k-1}.  
\end{equation}
Given a length $a$ vector of
nonnegative integers ${\bf r} = (r_1, r_2, \dots, r_a)$ satisfying $\sum i r_i = a$, the number of $a,b$-Dyck paths
with $r_i$ vertical runs of size $i$ for $1 \leq i \leq a$ is given by the {\em rational Kreweras number}
\begin{equation}
\Krew(a, b; {\bf r}) := \frac{1}{b} {b \choose r_1, r_2, \dots, r_a, b-k} = \frac{(b-1)!}{r_1! r_2! \cdots r_a! (b-k)!},
\end{equation} 
where 
$k = \sum r_i$ is the total number of vertical runs.  For example, the $5,7$-Dyck path shown in 
Figure~\ref{fig:homognc} contributes to $\Krew(5,7; {\bf r})$, where
${\bf r} = (1,2,0,0,0)$.

\subsection{Noncrossing partitions}
A set partition $\pi$ of $[n]$ is called {\em noncrossing} if, for all indices $1 \leq i < j < k < \ell \leq n$, we have that
$i \sim k$ in $\pi$ and $j \sim \ell \in \pi$ together imply that $i \sim j \sim k \sim \ell$ in $\pi$.
Equivalently, the set partition $\pi$ is noncrossing if and only if the convex hulls of the blocks of $\pi$ do not intersect
when drawn on the disk with boundary points labeled clockwise with $1, 2, \dots, n$.

We let $\NC(n)$ denote the collection of noncrossing partitions of $[n]$.
The {\em rotation} operator $\rot$ acts on the index set $[n]$, the power set $2^{[n]}$, and 
the collection $\NC(n)$ by the permutation
$\begin{pmatrix} 1 & 2 & \dots & n-1 & n \\ 2 & 3 & \dots & n & 1 \end{pmatrix}$.  These three sets also carry an action of the
{\em reflection} operator $\rfn$ by the permutation
$\begin{pmatrix} 1 & 2 & \dots & n-1 & n \\ n & n-1 & \dots & 2 & 1 \end{pmatrix}$.  Together, $\rot$ and $\rfn$ generate a dihedral
action on these sets.

\subsection{Rational noncrossing partitions}
In \cite{ARW}, rational Dyck paths were used to construct a rational generalization of the noncrossing partitions.  Let $D$
be an $a,b$-Dyck path and let $P \neq (0,0)$ be a lattice point which is at the bottom of a north step of $D$.  The {\em laser}
$\ell(P)$ 
 is the line segment 
of slope $\frac{a}{b}$ which is `fired' northeast
 from $P$ and continues until it intersects 
the Dyck path $D$.  By coprimality, the east endpoint of $\ell(P)$ is necessarily on the interior of an east step of $D$.

\begin{figure}
\includegraphics[scale = 0.25]{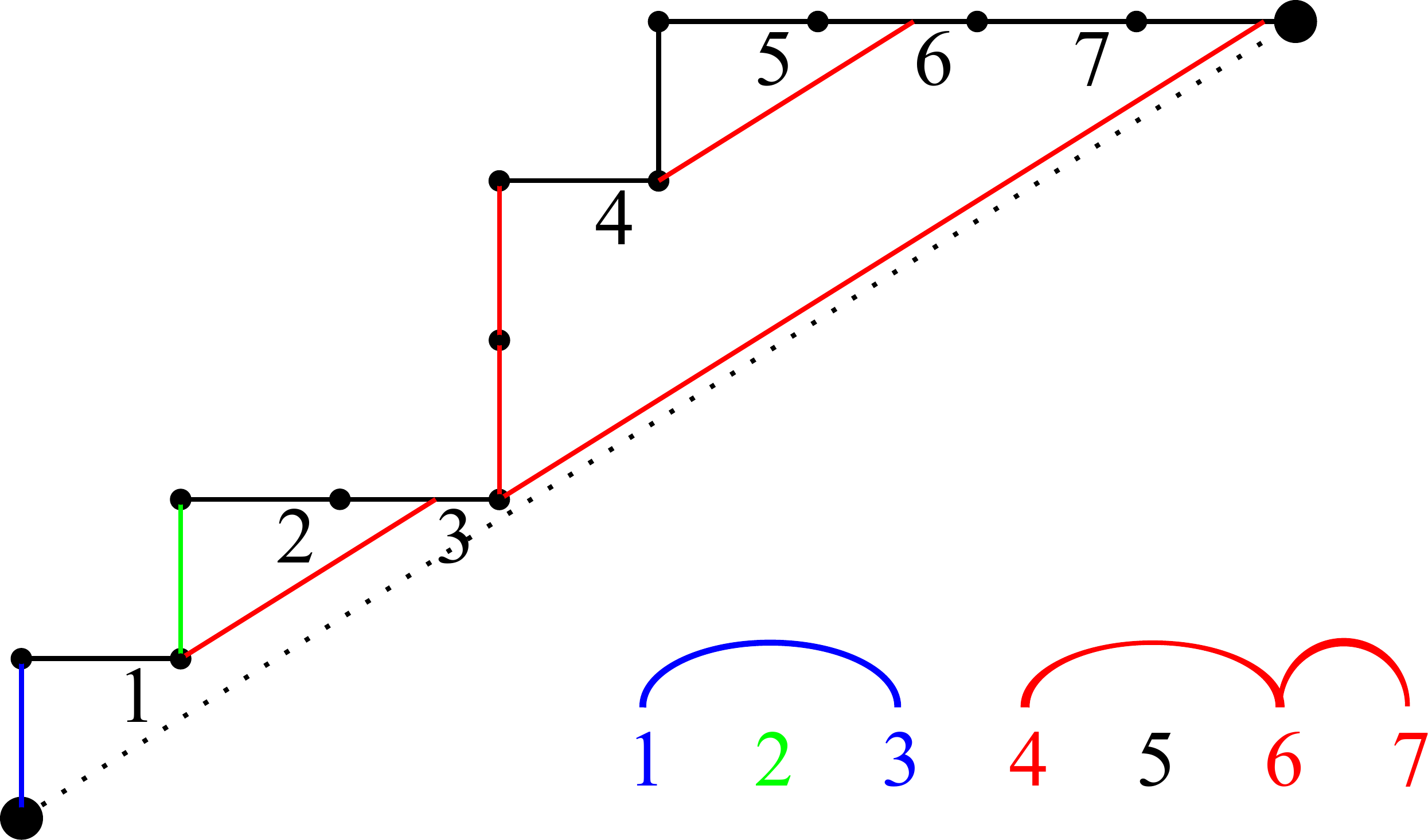}
\caption{A $5,8$-Dyck path and the corresponding $5,8$-noncrossing partition.  The visibility bijection
is shown with colors.}
\label{fig:nc-example}
\end{figure}

Let $D$ be an $a,b$-Dyck path.  We define a set partition $\pi(D)$ of $[b-1]$ as follows.  Label the east 
ends of the non-terminal east
steps of $D$ from left to right with $1, 2, \dots, b-1$ and fire lasers from all of the valleys of $D$. 
The set partition $\pi(D)$ is defined by the `visibility' relation
\begin{center}
$i \sim j$ if and only if the labels $i$ and $j$ are not separated by laser fire.
\end{center}
(Here we consider labels to lie slightly {\em below} their lattice points.) 
By construction, the set partition $\pi(D)$ is noncrossing.

An example of this construction when $(a,b) = (5,8)$ is shown in Figure~\ref{fig:nc-example}.  
If $D$ is an $a,b$-Dyck path, we have a natural `visibility' bijection from the set of vertical runs of $D$ to
the set of blocks of $\pi(D)$ which associates a vertical run at $y = i$ to the block of $\pi(D)$ whose minimum
element is $i+1$.  The visibility bijection is shown with colors in Figure~\ref{fig:nc-example}.

It will be convenient to think of the lasers fired from the valleys of a Dyck path $D$ in terms of their endpoints.
We let the {\em laser set} $\ell(D)$ of the $a,b$-Dyck path $D$ be the set of pairs $(i, j)$ such that $D$ contains a laser 
starting from a valley with $x$-coordinate $i$ and ending in the interior of an east step with west $x$-coordinate $j$.
For the $5,8$-Dyck path $D$ shown in Figure~\ref{fig:nc-example}, we have that
\begin{equation*}
\ell(D) = \{ (1,2), (3,7), (4,5) \}.
\end{equation*}
For $a < b$ coprime,
we define 
the  set of {\em admissible lasers}
\begin{equation*}
A(a,b) := \left\{(i, j) \,:\, 1 \leq i < j \leq b-1 \text{ and } j - i = \left\lfloor \frac{rb}{a} \right\rfloor \text{ for some $r = 1, 2, \dots, a-1$ } \right\}.
\end{equation*}
Slope considerations show that $(i, j) \in \ell(D)$ for some $a,b$-Dyck path $D$ if and only if 
$(i, j) \in A(a,b)$.

By considering $\pi(D)$ for all possible $a,b$-Dyck paths $D$, we get the set of {\em $a,b$-noncrossing partitions}
\begin{equation*}
\NC(a,b) := \{\pi(D) \,:\, \text{$D$ an $a,b$-Dyck path}\} 
\end{equation*}
(This is called the set of
{\em inhomogeneous $a,b$-noncrossing partitions} in \cite{ARW}.)
It is clear from construction that $\NC(a,b) \subseteq \NC(b-1)$.
Some basic facts about $a,b$-noncrossing partitions are as follows.

\begin{proposition}
\label{nc-basic-facts}
Let $a < b$ be coprime positive integers.
\begin{enumerate}
\item The map
\begin{equation*}
\pi: \{ \text{$a,b$-Dyck paths} \} \rightarrow \NC(a,b)
\end{equation*}
is bijective, so that $|\NC(a,b)| = \Cat(a,b)$ and 
the number of $a,b$-noncrossing partitions with $k$ blocks is the  rational Narayana number $\Nar(a,b;k)$.
\item  If $\pi \in \NC(a,b)$ and $\pi'$ is a noncrossing partition of $[b-1]$ which coarsens $\pi$, then $\pi' \in \NC(a,b)$.
\end{enumerate}
\end{proposition}

When $(a,b) = (n,n+1)$, the set $\NC(n,n+1)$ of rational noncrossing partitions is just the set $\NC(n)$ of all noncrossing partitions of 
$[n]$.  When $(a,b) = (n,mn+1)$ we have that $\NC(n, mn+1)$ of rational noncrossing partitions is the set of 
all noncrossing partitions of $[mn]$ which are {\em $m$-divisible} in the sense that every block size is divisible by $m$.
Armstrong, Rhoades, and Williams posed the problem of finding an analogous `intrinsic' characterization
 of $\NC(a,b)$ for arbitrary $a < b$ coprime.  We give two such characterizations in 
 Section~\ref{Characterizations}.

\subsection{Homogeneous rational noncrossing partitions}
Rational Dyck paths are used in \cite{ARW} to construct a generalization of noncrossing perfect matchings on the set $[2n]$.
The construction is  similar to that of the rational noncrossing partitions.

\begin{figure}
\includegraphics[scale = 0.5]{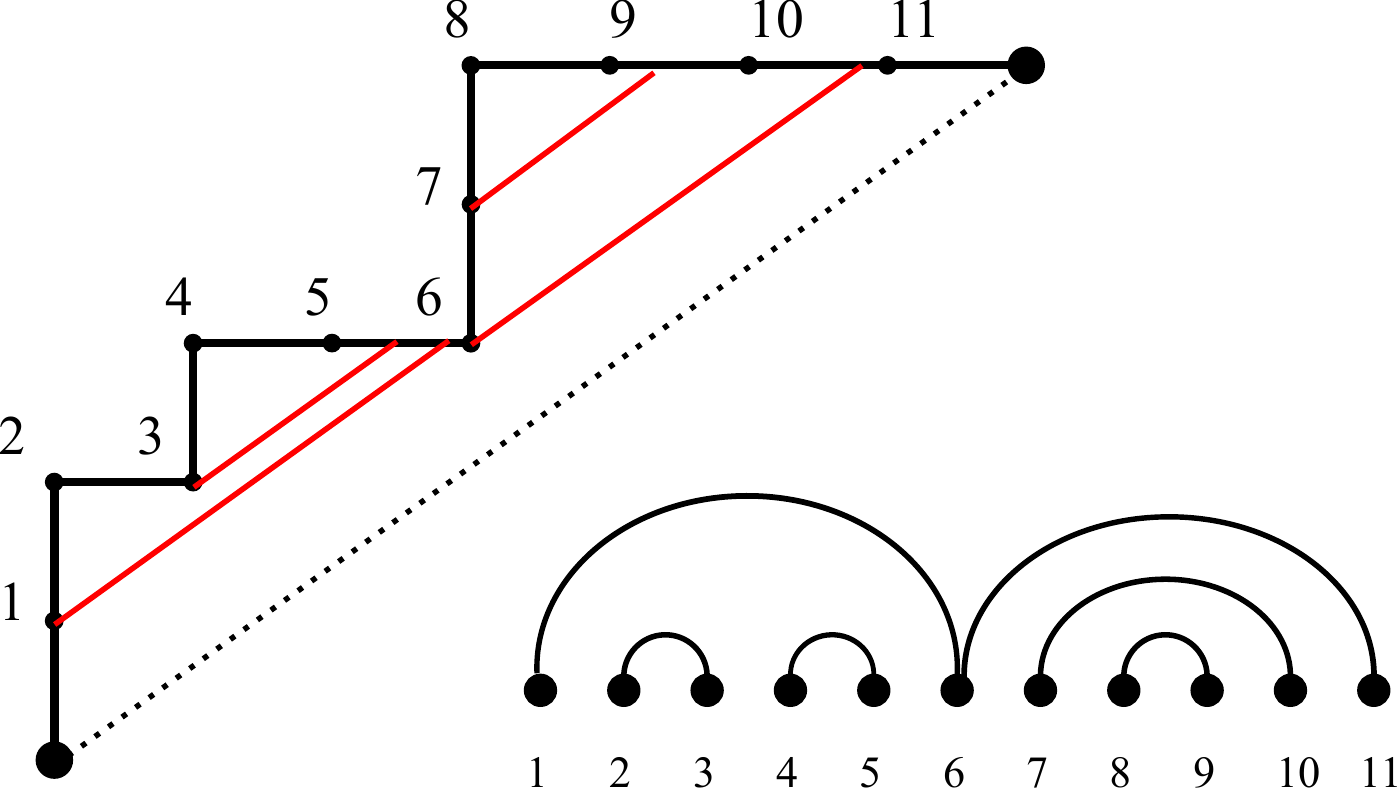}
\caption{A $5,7$-Dyck path and the corresponding $5,7$-homogeneous noncrossing partition.}
\label{fig:homognc}
\end{figure}

Let $(a,b)$ be coprime and let $D$ be an $a,b$-Dyck path.  We construct a set partition $\overline{\pi}(D)$ of the set
$[a + b - 1]$ as follows.  Label the interior lattice points of $D$ from southwest to northeast with $1, 2, \dots, a+b-1$.
Fire lasers of slope $\frac{a}{b}$ from 
{\em every} lattice point of $D$ which is at the south end of a north step (not just those lattice points which 
are valleys of $D$).  
The set partition $\overline{\pi}(D)$ is given by
declaring $i \sim j$ if and only if  the labels $i$ and $j$ are not separated by laser fire.  As before, we consider labels to
be slightly below their lattice points. Topological considerations make it clear
 that $\overline{\pi}(D)$ is a noncrossing partition of $[a+b-1]$.
 An example of this construction is shown in Figure~\ref{fig:homognc} for $(a,b) = (5,8)$.

Considering $\overline{\pi}(D)$ for all possible $a,b$-Dyck paths $D$ gives rise to the set of
{\em $a,b$-homogeneous rational noncrossing partitions}
\begin{equation*}
\HNC(a,b) := \{ \overline{\pi}(D) \,:\, \text{$D$ an $a,b$-Dyck path} \}.
\end{equation*} 
The adjective `homogeneous' refers to the fact that every set partition in $\HNC(a,b)$ has $a$ blocks.
By construction, we have that $\HNC(a,b) \subseteq \NC(a+b-1)$. 

 When $(a,b) = (n, n+1)$, the set 
$\HNC(n,n+1)$ is the set of noncrossing perfect matchings on $[2n]$.  When $(a, b) = (n, mn+1)$, the set
$\HNC(n, mn+1)$ is the set of noncrossing set partitions on $[mn]$ in which every block has size $m$ 
(these are also called {\em $m$-equal} noncrossing partitions).
Some basic facts about $\HNC(a,b)$ for general $a < b$ are as follows.

\begin{proposition}
\label{hnc-basic-facts}
Let $a < b$ be coprime positive integers.
\begin{enumerate}
\item The map
\begin{equation*}
\overline{\pi}: \{ \text{$a,b$-Dyck paths} \} \rightarrow \HNC(a,b)
\end{equation*}
is bijective, so that $|\HNC(a,b)| = \Cat(a,b)$.
\item The set $\HNC(a,b)$  is closed under the rotation operator $\rot$.
\end{enumerate}
\end{proposition}

In \cite{ARW} a promotion operator on $a,b$-Dyck paths is shown to intertwine the action of 
$\rot$ with the map $\overline{\pi}$ from paths to homogeneous noncrossing partitions.  
We will give an inhomogeneous analog of this promotion operator in the next section.

\section{Characterizations of Rational Noncrossing Partitions}
\label{Characterizations}

\subsection{The action of rotation}
Let $a < b$ be coprime.  In this subsection we will prove that $\NC(a,b)$ is closed under the rotation operator $\rot$
by describing rotation (or, rather, its inverse) as an operator on $a,b$-Dyck paths.
We will define an operator
\begin{equation*}
\rot': \{ \text{$a,b$-Dyck paths} \} \longrightarrow \{ \text{$a,b$-Dyck paths} \}
\end{equation*}
 on the set of $a,b$-Dyck paths
 which satisfies $\rot^{-1} \circ \pi = \pi \circ \rot'$.
 
 \begin{figure}
\includegraphics[scale = 0.7]{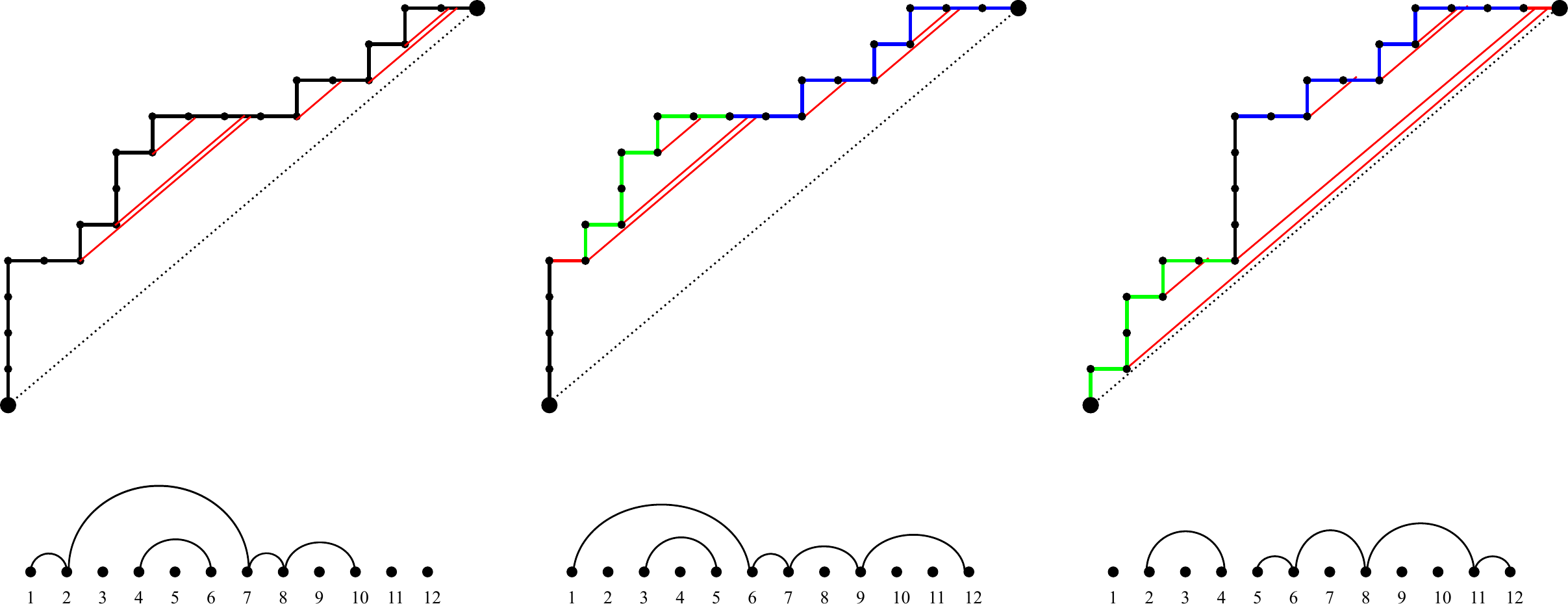}
\caption{The action of $\rot'$ on $a,b$-Dyck paths.}
\label{fig:rotation}
\end{figure}

The definition of $\rot'$ is probably best understood visually.
Figure~\ref{fig:rotation} shows three $11,13$-Dyck paths.  
The middle path is the image of the left path under
$\rot'$ and the right path is the image of the middle path under $\rot'$.  
The valley lasers are fired on each of these Dyck paths, and the corresponding 
partitions in $\NC(11,13)$ are shown below, but the vertex labels on the Dyck paths
are omitted for  legibility.

Let $D_1$ be the $11,13$-Dyck path on the left of Figure~\ref{fig:rotation}.  The 
westernmost horizontal
run of $D_1$ has size $> 1$.  Because of this, we define $D_2 := \rot'(D_1)$ by removing one step from
this horizontal run and adding it to the easternmost horizontal run of $D_1$.  The lattice 
path $D_2$ is  displayed in the middle of Figure~\ref{fig:rotation}.  Informally, the path $D_2$ is obtained 
from $D_1$ by translating one unit west.

Let $D_2$ be the $11,13$-Dyck path in the middle of Figure~\ref{fig:rotation}.  The westernmost
horizontal run of $D_2$ has size $1$.  The definition of $D_3 := \rot'(D_2)$ is more complicated in this case.
We break the lattice path $D_2$ up into four subpaths.  The first subpath, shown in black, is the initial northern
run of $D_2$.  The second subpath, shown in red, is the single east step which occurs after this run.  
The third subpath, shown in green, extends from the westernmost valley of $D_2$ to the (necessarily east)
step just before the laser fired from this valley hits $D_2$.  The fourth subpath, shown in blue, extends from
the end of the third subpath to the terminal point $(b,a)$ of $D_2$.  The path $D_3$ is obtained
by concatenating the third, first, fourth, and second subpaths, in that order.  The concatenation
of  green, then black, then blue, then red is shown on the right of Figure~\ref{fig:rotation}.

 More formally,
 given an $a,b$-Dyck path $D$, the definition of $\rot'(D)$ breaks up into three cases.
 Let $D = N^{i_1} E^{j_1} \cdots N^{i_m} E^{j_m}$ be the decomposition of $D$ into nonempty 
 vertical and horizontal runs.
\begin{enumerate}
\item
If $m = 1$ so that $D = N^a E^b$, we set 
\begin{equation*}
\rot'(D) := N^a E^b = D.
\end{equation*}
\item
If $m, j_1 > 1$,  we set
\begin{equation*}
\rot'(D) := N^{i_1} E^{j_1 - 1} N^{i_2} E^{j_2} \cdots N^{i_m} E^{j_m + 1}.
\end{equation*}
\item
If $m > 1$ and $j_1 = 1$, let $P = (1, i_1)$ be the westernmost valley of $D$.  The laser $\ell(P)$ fired from $P$
hits $D$ on a horizontal run $E^{i_k}$ for some $2 \leq k \leq m$.  Suppose that $\ell(P)$ hits the horizontal run
$E^{i_k}$ on step $r$, where $2 \leq r \leq i_k$.  We set
\begin{equation*}
\rot'(D) := N^{i_2} E^{i_2} \cdots N^{i_{k-1}} E^{i_{k-1}} N^{i_k} E^{r-1} N^{i_1} E^{i_k - r + 1} N^{i_{k+1}} E^{i_{k+1}}  \cdots N^{i_m} E^{j_m + 1}.
\end{equation*}
\end{enumerate}

\begin{proposition}
\label{closed-under-rotation}
The definition of $\rot'$ given above gives a well defined operator on the set of $a,b$-Dyck paths.
As operators $\{ a,b$-Dyck paths $\} \longrightarrow \NC(a,b)$, we have that $\rot^{-1} \circ \pi = \pi \circ \rot'$.
In particular, the set $\NC(a,b)$ is closed under rotation.
\end{proposition}

\begin{proof}
Let $D$ be an $a,b$-Dyck path.  If the westernmost horizontal run of $D$ has size $> 1$, it is clear that 
$\rot'(D)$ is also an $a,b$-Dyck path and that the corresponding set partitions are related by
$\rot^{-1}(\pi(D)) = \pi(\rot'(D))$.  We therefore
assume that the westernmost horizontal run of $D$ consists of a single step.  

We claim that the lattice path
$\rot'(D)$ stays above the diagonal $y = \frac{a}{b} x$.  Indeed, consider the decomposition of 
$\rot'(D)$ as in Figure~\ref{fig:rotation}.  The first (green) subpath of $\rot'(D)$ stays above $y = \frac{a}{b}x$
because the laser fired from the westernmost valley of $D$ has slope $\frac{a}{b}$.
The second (black) subpath of $\rot'(D)$ is a vertical run, 
so the concatenation of the first and second subpaths 
of $\rot'(D)$ stay above $y = \frac{a}{b} x$.  The third (blue) subpath of $\rot'(D)$ is just the 
corresponding subpath of $D$ translated one unit west, and certainly stays above the line $y = \frac{a}{b} x$.
Since the fourth (red) subpath of $\rot'(D)$ is a single east step, we conclude that the entire path
$\rot'(D)$ stays above $y = \frac{a}{b} x$, and so is an $a,b$-Dyck path.

By the last paragraph, the set partition $\pi(\rot'(D)) \in \NC(a,b)$ is well defined.  We argue that 
$\rot^{-1}(\pi(D)) = \pi(\rot'(D))$.  To do this, we consider how the valley lasers of $\rot'(D)$ relate
to the corresponding valley lasers of $D$.  
\begin{itemize}
\item  The valley lasers in the third (blue) subpath of $\rot'(D)$ are just the valley lasers
 in the corresponding subpath of $D$ shifted one unit west.
\item  The valley lasers in the first (green) subpath of $\rot'(D)$ are either the valley lasers in the
corresponding subpath of $D$ shifted one unit west, or hit $\rot'(D)$ on its terminal east step, depending
on whether these lasers hit $D$ in its third (green) or fourth (blue) subpaths.
\item  The valley laser of the westernmost valley in $D$ is replaced by the laser in the valley of $\rot'(D)$ between
its first (green) and second (black) subpaths, which necessarily hits $\rot'(D)$ on its terminal east step. 
\end{itemize}
From this description of the valley lasers of $\rot'(D)$, one checks that $\rot^{-1}(\pi(D)) = \pi(\rot'(D))$,
as desired.
\end{proof}

\subsection{Characterization from Kreweras complement}
Our first characterization of rational noncrossing partitions gives a  description of their Kreweras complements.
Let $\pi$ be a noncrossing partition of $[n]$.  
The {\em Kreweras complement} $\krew(\pi)$ is the noncrossing partition obtained by drawing the $2n$ vertices
$1, 1', 2, 2', \dots, n, n'$ clockwise on the boundary of a disk in that order, drawing the blocks of 
$\pi$ on the unprimed vertices, and letting $\krew(\pi)$ be the unique coarsest partition of the primed vertices which introduces
no crossings.
The map $\krew: \NC(n) \rightarrow \NC(n)$ satisfies $\krew^2 = \rot$, and so is a bijection.

\begin{figure}
\includegraphics[scale = 0.3]{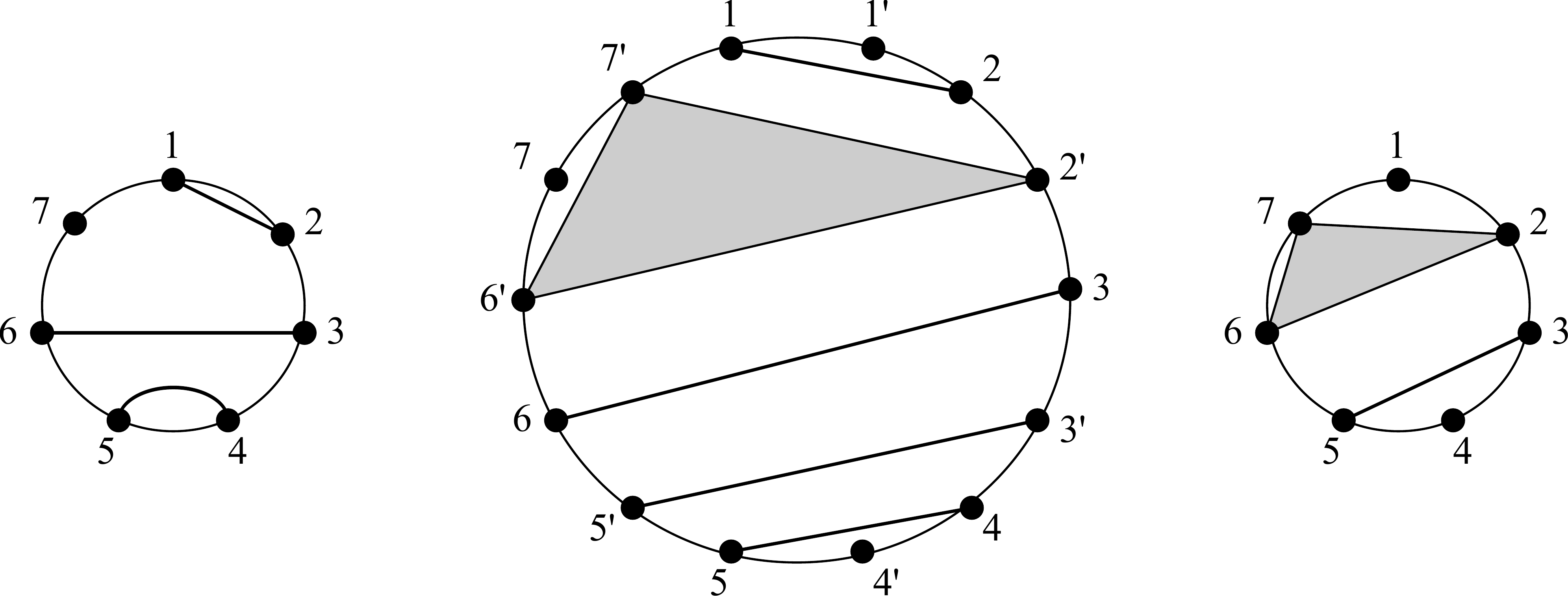}
\caption{An example of Kreweras complement.}
\label{fig:kreweras}
\end{figure}

An example of Kreweras complementation for $n = 7$ is shown in Figure~\ref{fig:kreweras}.  We have that
$\krew: \{ \{1,2\}, \{3,6\}, \{4,5\}, \{7\} \} \mapsto \{ \{1\}, \{2,6,7\}, \{3,5\}, \{4\} \}$. 

For $(a, b) \neq (n, n+1)$, the set $\NC(a,b)$  is {\em not} closed under Kreweras complement.
Indeed, the one block set partition $\{ [b-1] \}$ is contained in $\NC(a,b)$ but its Kreweras complement
(the all singletons set partition) is not.  
 On the other hand, 
the set $\NC(a,b)$ is in bijective correspondence with its Kreweras image $\krew(\NC(a,b))$.
The problem of characterizing $\NC(a,b)$ is therefore equivalent to the problem of characterizing its 
Kreweras image.

Given $\pi \in \NC(a,b)$, 
there is a simple relationship between the blocks of $\krew(\pi) \in \NC(b-1)$ and the laser set $\ell(D)$ of the 
$a,b$-Dyck path $D$ corresponding to $\pi$.  

\begin{lemma}
\label{kreweras-lasers}
Let $a < b$ be  coprime, let $\pi \in \NC(a,b)$ have corresponding $a,b$-Dyck path $D$, and let 
$\krew(\pi)$ be the Kreweras complement of $\pi$.  The laser set $\ell(D)$ is 
\begin{equation*}
\ell(D) = \{ (i, \max(B)) \,:\, \text{$B$ is a block of $\krew(\pi)$ and $i$ is a nonmaximal element of $B$} \}.
\end{equation*}
\end{lemma}

\begin{proof}
This is clear from the topological definition of $\pi(D)$ and Kreweras complement.
\end{proof}

Let $\pi \in \NC(b-1)$ be a noncrossing partition and suppose we would like to determine whether $\pi$
is in $\NC(a,b)$.
In light of Lemma~\ref{kreweras-lasers}, it is enough to know whether there exists an $a,b$-Dyck path $D$ whose laser
set consists of pairs $(i, \max(B))$, where $i \in B$ runs over all nonmaximal elements of all blocks $B$ of
$\krew(\pi)$.
A characterization of the laser sets of $a,b$-Dyck paths can be read off from the results of \cite{RhoadesAlexander}.
In \cite{RhoadesAlexander} this characterization was used to prove that the rational analog of the 
associahedron is a flag simplicial complex.

\begin{proposition}
\label{kreweras-characterization}
Let $a < b$ be  coprime and let $\pi$ be a noncrossing partition of $[b-1]$.  We have that $\pi$ is an $a,b$-noncrossing partition
if and only if for every block $B$ of the Kreweras complement $\krew(\pi)$ we have that
\begin{itemize}
\item  $(i, \max(B)) \in A(a,b)$ for every nonmaximal element $i \in B$, and
\item for any two nonmaximal elements $i < j$ in $B$,
\begin{equation*}
\left\lceil (\max(B) - i) \frac{a}{b} \right\rceil - \left\lceil (\max(B) - j) \frac{a}{b} \right\rceil > (j - i) \frac{a}{b}.
\end{equation*}
\end{itemize}
\end{proposition}

\begin{proof}
As mentioned above, we know that $\pi \in \NC(a,b)$ if and only if there exists an $a,b$-Dyck path $D$ 
whose laser set consists of all possible pairs $(i, \max(B))$, where $B$ ranges over all 
blocks of $\krew(\pi)$ and $i$ ranges over all nonmaximal elements of $B$.  
For such an $a,b$-Dyck path $D$ to exist, it is certainly necessary that $(i, \max(B)) \in A(a,b)$ always, so we assume this
condition holds.

By \cite[Proposition 1.3]{RhoadesAlexander} and
\cite[Lemma 4.3]{RhoadesAlexander}, we know that  an $a,b$-Dyck path $D$ as in the previous paragraph 
exists if and only if
for every two  pairs $(i, \max(B))$ and $(i', \max(B'))$, there exists an $a,b$-Dyck path
$\overline{D}$ with laser set $\ell(\overline{D})  = \{(i, \max(B)), (j, \max(B')) \}$.
By \cite[Lemma 5.3]{RhoadesAlexander} and \cite[Proposition 5.5]{RhoadesAlexander}, such an
$a,b$-Dyck path $\overline{D}$ does {\em not} exist if and only if $\max(B) = \max(B')$ (so that $B = B'$) and 
$\left\lceil (\max(B) - i) \frac{a}{b} \right\rceil - \left\lceil (\max(B) - j) \frac{a}{b} \right\rceil < (j - i) \frac{a}{b}$, where
$i < j$.  Since $1 \leq i < j < b-1$, 
we have
$\left\lceil (\max(B) - i) \frac{a}{b} \right\rceil - \left\lceil (\max(B) - j) \frac{a}{b} \right\rceil \neq (j - i) \frac{a}{b}$
and the result follows.
\end{proof}

\begin{example}
Let $(a,b) = (5,8)$ and consider the following three partitions in $\NC(7)$:
\begin{align*}
\pi_1 &:= \{ \{1\}, \{2,6,7\}, \{3\}, \{4,5\} \}, \\
\pi_2 &:= \{ \{1\}, \{2,6,7\}, \{3,4\}, \{5\} \}, \\
\pi_3 &:= \{ \{1\}, \{2,7\}, \{3,4,5\}, \{6\} \}.
\end{align*}
To determine which of $\pi_1, \pi_2, \pi_3$ belong to $\NC(5,8)$, we start by computing their Kreweras images:
\begin{align*}
\krew(\pi_1) &:= \{ \{1,7\}, \{2,3,5\}, \{4\}, \{6\} \}, \\
\krew(\pi_2) &:= \{ \{1,7\}, \{2,4,5\}, \{3\}, \{6\} \}, \\
\krew(\pi_3) &:= \{ \{1,7\}, \{2,5,6\}, \{3\}, \{4\} \}.
\end{align*}
Since $\{2,3,5\}$ is a block of $\krew(\pi_1)$ and $(3,5) \notin A(5,8)$, we conclude that $\pi_1 \notin \NC(5,8)$.
Since $\{2,4,5\}$ is a block of $\krew(\pi_2)$ and 
$\left \lceil (5-2) \frac{5}{8} \right \rceil - \left \lceil (5-4) \frac{5}{8} \right \rceil \leq (4-2) \frac{5}{8},$
we conclude that $\pi_2 \notin \NC(5,8)$.  
Since both of the bullets in Proposition~\ref{kreweras-characterization} hold for 
$\pi_3$, we conclude that $\pi_3 \in \NC(5,8)$.
\end{example}

\subsection{$a,b$-ranks}
In order to state our second characterization of $a,b$-noncrossing partitions, we introduce a new
way to measure the `size' of a block $B$ of a noncrossing partition $\pi$ other than its cardinality $|B|$.
Our rational analog of block size is as follows.

\begin{defn}
\label{rank-definition}
Let $a < b$ be coprime positive integers and let $\pi \in \NC(b-1)$ be a noncrossing partition of $[b-1]$.
We assign an integer $\rank^{\pi}_{a,b}(B) \in \ZZ$ to every block $B$ of $\pi$ 
 (or just $\rank(B)$ when $a, b,$ and $\pi$ are clear from context)
by the following recursive procedure.  Let $\preceq$ be the partial order on the blocks of $\pi$ defined
by 
\begin{equation*}
B' \preceq B \Longleftrightarrow [\min(B'), \max(B')] \subseteq [\min(B), \max(B)].
\end{equation*}
The integers $\rank^{\pi}_{a,b}(B)$ are implicitly determined by the  formula
\begin{equation}
\sum_{B' \preceq B} \rank^{\pi}_{a,b}(B') = \left \lceil(\max(B) - \min(B) + 1)\frac{a}{b} \right \rceil,
\end{equation}
for all blocks $B \in \pi$.
\end{defn}

While we define $\rank^{\pi}_{a,b}(\cdot)$ as a function on the blocks of an arbitrary noncrossing 
partition of $[b-1]$, this notion of rank will be more useful for $a,b$-noncrossing partitions. 
We will see that rank is more useful than cardinality as a block size measure in rational Catalan theory.
This subsection proves basic properties of the function $\rank^{\pi}_{a,b}(\cdot)$ on the blocks of 
$a,b$-noncrossing partitions.

In the Fuss-Catalan case $b \equiv 1$ (mod $a$), rank and cardinality are equivalent.

\begin{proposition}
\label{ranks-fuss}
Let $(a, b) = (n, mn+1)$,
let $\pi \in \NC(n,mn+1)$ be an $m$-divisible noncrossing partition of $[mn]$, and let $B$ be a block of $\pi$.  We have that
\begin{equation*}
\rank^{\pi}_{n,mn+1}(B) = \frac{|B|}{m}.
\end{equation*}
\end{proposition} 

\begin{proof}
For any block $B$ of $\pi$, we have that $1 \leq \min(B) \leq \max(B) \leq mn$ and the divisibility relation
$m | (\max(B) - \min(B) + 1)$, so that 
\begin{equation*}
 \left\lceil (\max(B) - \min(B) + 1)\frac{n}{mn+1} \right\rceil = \frac{\max(B) - \min(B) + 1}{m}.
\end{equation*}
The assignment $B \mapsto \frac{|B|}{m}$ therefore satisfies the recursion for 
$\rank^{\pi}_{n,mn+1}(\cdot)$.
\end{proof}

For general $a < b$ coprime, the $a,b$-rank  of a block of an $a,b$-noncrossing partition
$\pi \in \NC(a,b)$ is {\em not} necessarily determined by its cardinality.  
For example, consider $(a,b) = (3,5)$.  Then $\pi = \{ \{1,3\}, \{2\}, \{4\} \} \in \NC(3,5)$.  We have that
\begin{equation*}
\rank^{\pi}_{3,5}(\{1,3\}) = \rank^{\pi}_{3,5}(\{2\}) = \rank^{\pi}_{3,5}(\{4\}) = 1.
\end{equation*}
Proposition~\ref{ranks-fuss} shows that the divergence
between rank and cardinality  is a genuinely new feature of rational Catalan combinatorics
and is invisible at the Catalan and Fuss-Catalan levels of generality.

\begin{defn}
If $\pi \in \NC(a,b)$ is an $a,b$-noncrossing partition,  the {\em rank sequence} $R(\pi) = (r_1, \dots, r_{b-1})$ 
of $\pi$ is the sequence of nonnegative integers given by
\begin{equation*}
r_i := 
\begin{cases}
\rank(B) & \text{if $i = \min(B)$ for some block $B$ of $\pi$,} \\
0 & \text{otherwise.}
\end{cases}
\end{equation*}
\end{defn}

For example, the rank sequence of the $5,8$-noncrossing partition shown in Figure~\ref{fig:nc-example}
has rank sequence $(1,1,0,2,1,0,0)$.
Observe that this is also the sequence of vertical runs of the corresponding $5,8$-Dyck path, so that the rank
sequence is equivalent to the Dyck path.  
This is a general phenomenon.

\begin{proposition}
\label{recover-dyck-path}
Let $a < b$ be  coprime, let $\pi$ be an $a,b$-noncrossing partition with rank sequence 
$R(\pi) = (r_1, \dots, r_{b-1})$, and let $D$ be the $a,b$-Dyck path associated to $\pi$.

\begin{enumerate}
\item
For any block $B$ of $\pi$, the vertical run of $D$ visible from $\pi$ has size $\rank(B)$.
\item
The Dyck path $D$ is given by
\begin{equation*}
D = N^{r_1} E N^{r_2} E \cdots N^{r_{b-1}}E E.
\end{equation*}
\end{enumerate}
\end{proposition}

\begin{proof}
For any block $B$ of $\pi$, we have that $\min(B)$ labels the lattice point just
to the right of 
the vertical run labeled by $B$.  Therefore, (2) follows from (1).

To prove (1), recall the partial order $\preceq$ on the blocks of $\pi$.  If $B$ is minimal with respect to
$\preceq$, then $B = [\min(B), \max(B)]$ is an interval.  The labels of $B$ must appear on a single 
horizontal run of $D$, say on the line $y = c$.  Let $i_B$ be the length of the vertical run visible from $B$,
let $P$ be the valley at the bottom of this vertical run, and let $\ell(P)$ be the laser fired from $P$.

The laser $\ell(P)$ must intersect the line $y = c$.  We claim that it does so in the open $x$-interval
$(\max(B), \max(B) + 1)$.  Indeed, by coprimality there exists a unique lattice point $Q$ on $D$
which is northwest of the laser $\ell(P)$ and  has minimum horizontal distance to the laser $\ell(P)$.
Let $m$ be the $x$-coordinate of $Q$.  We claim that $m \in B$.  Indeed, if $m \notin B$, there must be 
a laser $\ell(Q')$ fired from a valley $Q'$ which separates $m$ from $B$. But then $Q'$ would be 
closer than $Q$ to $\ell(P)$, so we conclude that $m \in B$.  
Since $B = [\min(B), \max(B)]$ is an interval, we have that $\ell(P)$ intersects $y = c$ in the open 
$x$-interval $(\max(B), \max(B) + 1)$.  This implies that $\rank(B) = i_B$, as desired.

Now suppose that $B$ is not $\preceq$-minimal among the blocks of $\pi$.  Then the interval
$[\min(B), \max(B)]$ is a union of at least two blocks of $\pi$.   For any block $B'$ contained in this interval,
let $i_{B'}$ denote the size of the vertical run visible from $B'$.  We may inductively assume that,
for all $B' \neq B$, we have that
\begin{equation*}
\rank^{\pi}_{a,b}(B') = i_{B'}.
\end{equation*}

Let $\pi_0$ be the set partition obtained from $\pi$ by merging the blocks contained in $[\min(B), \max(B)]$.
Then $\pi_0$ is noncrossing, and hence
$a,b$-noncrossing by Proposition~\ref{nc-basic-facts} (2). The recursion for $\rank$
says that
\begin{equation*}
\rank^{\pi_0}_{a,b}([\min(B), \max(B)]) = \sum_{B' \subseteq [\min(B), \max(B)]} \rank^{\pi}_{a,b}(B)
= \rank^{\pi}_{a,b}(B) + \sum_{\substack{B' \subseteq [\min(B), \max(B)] \\ B' \neq B}} i_{B'},
\end{equation*}
where the second equality used the inductive hypothesis.  Moreover, the Dyck path $D_0$
corresponding to $\pi_0$ is obtained from $D$
by replacing the portion between the $x$-coordinates $\min(B)$ and
$\max(B)$ with a single vertical run, followed by a single horizontal run.  In particular, the size 
of the vertical run in $D_0$ visible from $[\min(B), \max(B)]$ is $\sum_{B' \subseteq [\min(B), \max(B)]} i_{B'}$.
This forces $\rank^{\pi}_{a,b}(B) = i_B$, as desired.
\end{proof}

The cardinality function $| \cdot |$ on blocks of noncrossing partitions satisfies the following two properties.
\begin{itemize}
\item  Let $\pi$ and $\pi'$ be noncrossing partitions such that $\pi$ refines $\pi'$.  For any block $B'$ 
of $\pi'$, write $B' = B_1 \uplus \cdots \uplus B_k$, where $B_1, \dots, B_k$ are blocks of $\pi$.  We have
\begin{equation*}
|B'| = |B_1| + \cdots + |B_k|.
\end{equation*} 
\item  The function $| \cdot |$ on blocks of noncrossing set partitions is invariant under the dihedral action
of $\langle \rot, \rfn \rangle$.
\end{itemize}
To  justify our intuition that rank measures size, 
we prove that $\rank^{\pi}_{a,b}(\cdot)$ enjoys the same properties on the set of $a,b$-noncrossing partitions.

\begin{proposition}
\label{ranks-add}
Let $a < b$ be coprime, let $\pi, \pi' \in \NC(a,b)$, and assume that $\pi$ refines $\pi'$.  Let $B'$ be a block
of $\pi'$ and let $B_1, \dots, B_k$ be the blocks of $\pi$ such that $B' = B_1 \uplus \cdots \uplus B_k$.  We have
\begin{equation*}
\rank^{\pi'}_{a,b}(B') = \rank^{\pi}_{a,b}(B_1) + \cdots + \rank^{\pi}_{a,b}(B_k).
\end{equation*}
\end{proposition}

\begin{proof}
It suffices to consider the case where $\pi'$ is obtained by merging two blocks of $\pi$, say 
$B_1$ and $B_2$. Suppose $\min(B_1) < \min(B_2)$.  If $D, D'$ are the 
$a,b$-Dyck paths corresponding to $\pi, \pi'$, then $D'$ is obtained from $D$ by moving the vertical 
run visible from $B_2$ on top of the vertical run visible from $B_1$.  The merged vertical run in $D'$ is visible
from $B_1 \uplus B_2$.
The result follows from Proposition~\ref{recover-dyck-path}.
\end{proof}

Proposition~\ref{ranks-add} fails for partitions $\pi, \pi \in \NC(b-1)$ which are not  both $a,b$-noncrossing.  
For example, let 
$(a,b) = (2,5)$, $\pi = \{ \{1, 2, 3\}, \{4\} \},$ and $\pi' = \{ \{1,2,3,4\} \}$.  We have
$\rank^{\pi}_{2,5}(\{1,2,3\}) = 2, \rank^{\pi}_{2,5}(\{4\}) = 1,$ and
$\rank^{\pi'}_{2,5}(\{1,2,3,4\}) = 2$.

\begin{proposition}
\label{dihedral-invariance}
Let $a < b$ be coprime, let $\pi \in \NC(a,b)$, and let $B$ be a block of $\pi$.  We have
\begin{align}
\rank^{\pi}_{a,b}(B) &= \rank^{\rot(\pi)}_{a,b}(\rot(B)), \\
\rank^{\pi}_{a,b}(B) &= \rank^{\rfn(\pi)}_{a,b}(\rfn(B)).
\end{align}
\end{proposition}

\begin{proof}
Since the intervals $[\min(B), \max(B)]$ and $[\min(\rfn(B)), \max(\rfn(B))]$ have the same length, 
reflection invariance is clear (and still holds if $\pi$ fails to be $a,b$-noncrossing).

Recall that $\rot$ acts by adding $1$ to every index, modulo $b-1$.
Rotation invariance is therefore clear unless $B$ has the form 
$B = \{ i_1 < \dots < i_j < b-1\}$ with $i_1 > 1$.  In this case, the intervals $[1, i_1 - 1], [i_2 + 1, i_3 - 1], \dots, [i_j + 1, b-2]$
must each be unions of blocks of $\pi$.  Let $B_1, \dots, B_r$ be a complete list of these remaining blocks.  We certainly have
$\rank^{\pi}_{a,b}(B_i) = \rank^{\rot(\pi)}_{a,b}(\rot(B_i))$ for $1 \leq  i \leq r$.
On the other hand, Proposition~\ref{closed-under-rotation} 
and the assumption that $\pi \in \NC(a,b)$ guarantee that 
$\rot(\pi) \in \NC(a,b)$.  Proposition~\ref{recover-dyck-path} shows that
\begin{align*}
a &= \rank^{\pi}_{a,b}(B) + \rank^{\pi}_{a,b}(B_1) + \cdots + \rank^{\pi}_{a,b}(B_r) \\
&= \rank^{\rot(\pi)}_{a,b}(\rot(B)) + \rank^{\rot(\pi)}_{a,b}(\rot(B_1)) + \cdots + \rank^{\rot(\pi)}_{a,b}(\rot(B_r)).
\end{align*}
Since $\rank^{\pi}_{a,b}(B_i) = \rank^{\rot(\pi)}_{a,b}(\rot(B_i))$ for $1 \leq  i \leq r$, we get
$\rank^{\pi}_{a,b}(B) = \rank^{\rot(\pi)}_{a,b}(\rot(B))$.
\end{proof}

Proposition~\ref{dihedral-invariance} does not hold for partitions $\pi \in \NC(b-1)$ which 
are not $a,b$-noncrossing.  For example, take $(a,b) = (2, 5)$ and $\pi = \{ \{1\}, \{2,3,4\} \}$, so that
$\rot(\pi) = \{ \{1,3,4\} \{2\} \}$.  We have that 
$\rank^{\pi}_{2,5}(\{2,3,4\}) = 2$ but $\rank^{\rot(\pi)}_{2,5}(\{1,3,4\}) = 1$.

\subsection{Characterization from rank function}
The $a,b$-rank function $\rank^{\pi}_{a,b}(\cdot)$ can be used to decide whether a given noncrossing partition 
$\pi \in \NC(b-1)$ is $a,b$-noncrossing.
This is our first application of $a,b$-rank to rational Catalan theory.
To begin, we introduce the notion of a valid $a,b$-ranking.

\begin{defn}
\label{def:valid-ranking}
Let $\pi \in \NC(b-1)$ be an arbitrary noncrossing partition of $[b-1]$.
We say that $\pi$ has a {\em valid $a,b$-ranking} if 
$\rank^{\pi}_{a,b}(B) > 0$ for every block $B \in \pi$ and
\begin{equation*}
\sum_{B \in \pi} \rank(B) = a.
\end{equation*}
\end{defn}

\begin{example}
\label{valid-example}
Let $(a, b) = (3, 5)$ and consider the three noncrossing partitions 
\begin{center}
$\pi_1 = \{ \{1, 2, 4\}, \{3\} \},
\pi_2 = \{ \{1, 4\}, \{2\}, \{3\} \},$ and $\pi_3 = \{ \{1, 2\}, \{3\}, \{4\}\}$.  
\end{center}
We have that $\pi_1$ has a valid
$3,5$-ranking, but $\pi_2$ and $\pi_3$ do not.  Indeed,
we have $\rank^{\pi_2}_{3,5}(\{1,4\}) = 0$ and 
$\rank^{\pi_3}_{3,5}(\{1,2\}) + \rank^{\pi_3}_{3,5}(\{3\}) + \rank^{\pi_3}_{3,5}(\{4\}) = 2 + 1 + 1 > 3$.
\end{example}

We record the following lemma for future use.

\begin{lemma}
\label{inequality}
Let $a < b$ be coprime and
let $\pi \in NC(b-1)$ be an arbitrary noncrossing partition of $[b-1]$.  We have 
\begin{equation}
\sum_{B \in \pi} \rank(B) \geq a.
\end{equation}
\end{lemma}

\begin{proof}
Recall the partial order $\preceq$ on the blocks of $\pi$.  If $B'$ is maximal with respect to $\preceq$, we have
\begin{equation*}
\sum_{B \preceq B'} \rank(B) = \left \lceil (\max(B') - \min(B') + 1)\frac{a}{b} \right \rceil > (\max(B') - \min(B') + 1)\frac{a}{b}.
\end{equation*}
Summing over all $\preceq$-maximal blocks $B'$ gives
\begin{equation}
\sum_{B \in \pi} \rank(B) > (b-1)\frac{a}{b},
\end{equation}
which is equivalent to the desired inequality since $\sum_{B \in \pi} \rank(B) \in \ZZ$.
\end{proof}

By Proposition~\ref{recover-dyck-path}, every $a,b$-noncrossing partition has a valid $a,b$-ranking. 
It would be nice if the converse held, but it does not; if $(a,b) = (2,5)$, the partition
$\pi = \{ \{1,2,4\}, \{3\}\}$ has a valid $2,5$-ranking but $\pi \notin \NC(2,5)$.
 Our second characterization of  $a,b$-noncrossing partitions 
states that $\pi \in \NC(a,b)$ if and only if every element 
in the orbit of $\pi$ under rotation has a 
valid $a,b$-ranking.  To prove this statement, we start by examining how  validity behaves under coarsening.  

In general, having a valid $a,b$-ranking is not closed under coarsening of noncrossing partitions.  For example, let 
$(a, b) = (5, 11)$ and consider the two partitions $\pi, \pi' \in \NC(10)$ given by
\begin{align*}
\pi &= \{ \{1, 6, 7, 8, 9, 10\}, \{2\}, \{3\}, \{4\}, \{5\} \},  \\
\pi' &= \{ \{1, 6, 7, 8, 9, 10\}, \{2, 5\}, \{3\}, \{4\} \}.  
\end{align*}
Then $\pi$ has a valid $5,11$-ranking but $\pi'$ does not.  
On the other hand, certain types of coarsening do preserve  validity.

\begin{lemma}
\label{valid-rankings-coarsen}
Let $a < b$ be coprime and let $\pi \in \NC(b-1)$ be a noncrossing partition such that
$\pi$ has a valid $a,b$-ranking.  Suppose that $\pi' \in \NC(b-1)$ is another noncrossing partition
 obtained from
$\pi$ by one of the following two operations:
\begin{enumerate}
\item  merging two blocks $B_0, B_1$ of $\pi$ with $\min(B_1) = \max(B_0) + 1$, or
\item merging two blocks $B_0, B_1$ of $\pi$ with $B_0 \prec B_1$.
\end{enumerate}

The noncrossing partition $\pi'$ has a valid $a,b$-ranking.
\end{lemma}

Observe that if $\pi, \pi' \in \NC(b-1)$ are such that $\pi$ refines $\pi'$ and $\pi'$ is a union of intervals,
then $\pi'$ can be obtained from $\pi$ by a sequence of coarsenings as described in 
Lemma~\ref{valid-rankings-coarsen}.

\begin{proof}
First assume that $\pi'$ is obtained from $\pi$ by replacing $B_0$ and $B_1$ with $B_0 \cup B_1$ as in Condition 1.
For $i = 0, 1$, let $r_i$ denote the sum of the ranks of the blocks $B$ of $\pi$ satisfying $B \prec B_i$.  By the definition of
rank, we have
\begin{equation}
\label{eq:i_zero_or_one}
\rank^{\pi}_{a,b}(B_i) + r_i = \left \lceil (\max(B_i) - \min(B_i) + 1)\frac{a}{b} \right \rceil
\end{equation}
for $i = 0, 1$.  Adding these two equations together and recalling that
$\lceil x \rceil + \lceil y \rceil - 1 \leq \lceil x + y \rceil \leq \lceil x \rceil + \lceil y \rceil$ for any $x, y \in \mathbb{R}$, we get
\begin{equation}
\label{valid-inequalities}
\rank^{\pi}_{a,b}(B_0) + \rank^{\pi}_{a,b}(B_1) - 1 \leq \rank^{\pi'}_{a,b}(B_0 \cup B_1)  \leq \rank^{\pi}_{a,b}(B_0) + \rank^{\pi}_{a,b}(B_1).
\end{equation}
Since $\pi$ has a valid $a,b$-ranking, we have  $\rank^{\pi}_{a,b}(B_i) > 0$ for $i = 0, 1$, so that
$\rank^{\pi'}_{a,b}(B_0 \cup B_1) > 0$.
Our analysis breaks up into two cases depending on whether $B_0 \cup B_1$ is a $\preceq$-maximal block of $\pi'$.

If $B_0 \cup B_1$ is a $\preceq$-maximal block of $\pi'$,
for every block $B$ of $\pi'$ with $B \neq B_0, B_1$ we have 
$\rank^{\pi'}_{a,b}(B) = \rank^{\pi}_{a,b}(B)$, so that 
the chain of inequalities in (\ref{valid-inequalities}) implies 
$\sum_{B \in \pi'} \rank^{\pi'}_{a,b}(B) \leq a$, so $\sum_{B \in \pi'} \rank^{\pi'}_{a,b}(B) = a$ by Lemma~\ref{inequality}.
Since every block of $\pi'$ has a positive rank, we conclude that $\pi'$ has a valid $a,b$-ranking.

If $B_0 \cup B_1$ is not a $\preceq$-maximal block of $\pi'$,
there exists a unique block $B_2 \in \pi'$ such that $B_2$ covers $B_0 \cup B_1$ in $\preceq$.  
The recursion for $\rank^{\pi'}_{a,b}(\cdot)$ gives
\begin{equation}
\label{b2-rank}
\rank^{\pi'}_{a,b}(B_0 \cup B_1) + r_0 + r_1 + r_2 + \rank^{\pi'}_{a,b}(B_2) = \left \lceil (\max(B_2) - \min(B_2) + 1)\frac{a}{b} \right \rceil,
\end{equation}
where $r_2$ is the sum of the ranks of all of the blocks $B$ of $\pi'$ satisfying $B \prec B_2$ but $B \not\preceq B_0, B_1$.
Combining (\ref{b2-rank}) with the inequalities in (\ref{valid-inequalities}), we see that
$\rank^{\pi'}_{a,b}(B_2) \geq \rank^{\pi}_{a,b}(B_2) > 0$. 
Since 
\begin{equation}
\sum_{\substack{B \in \pi \\ B \preceq B_2}} \rank^{\pi}_{a,b}(B) =  \left \lceil (\max(B_2) - \min(B_2) + 1)\frac{a}{b} \right \rceil =
\sum_{\substack{B' \in \pi' \\ B' \preceq B_2}} \rank^{\pi'}_{a,b}(B'),
\end{equation}
we have
$\rank^{\pi'}_{a,b}(B) = \rank^{\pi}_{a,b}(B)$ for all blocks $B \neq B_0, B_1, B_2$.
In particular, the rank of every block of $\pi'$ is positive and
$\sum_{B \in \pi'} \rank^{\pi'}_{a,b}(B) = a$.  
We conclude that $\pi'$ has a valid $a,b$-ranking.

Now assume 
that $\pi'$ is obtained from $\pi$ by replacing $B_0$ and $B_1$ with $B_0 \cup B_1$ as in Condition 2.
It follows that $B_1$ covers $B_0$ in the $\preceq$-order. By the recursion for rank,
\begin{equation}
\rank^{\pi'}_{a,b}(B_0 \cup B_1) = \rank^{\pi}_{a,b}(B_0) + \rank^{\pi}_{a,b}(B_1)
\end{equation}
and the ranks of all other blocks of $\pi'$ equal the corresponding  ranks in $\pi'$, so that $\pi'$ has a valid $a,b$-ranking.
\end{proof}

We are ready to prove our $a,b$-rank  characterization of $a,b$-noncrossing partitions.

\begin{proposition}
\label{rank-sequence-characterization}
Let $a < b$ be  coprime and let $\pi$ be a noncrossing partition of $[b-1]$.  We have that $\pi$ is an $a,b$-noncrossing
partition if and only if every partition in the orbit of $\pi$ under rotation has a valid 
$a,b$-ranking.
\end{proposition}

\begin{proof}
Suppose $\pi \in \NC(a,b)$.  By Proposition~\ref{closed-under-rotation}, 
we know that the rotation orbit of $\pi$ is contained in
$\NC(a,b)$, so that every partition in this orbit has a valid $a,b$-ranking.

For the converse, suppose that $\pi \in \NC(b-1) - \NC(a,b)$.  We argue that some partition in the rotation orbit of $\pi$ does not 
have a valid $(a,b)$-ranking.  

Consider the Kreweras complement $\krew(\pi)$.  If there is a block $B \in \krew(\pi)$ and an 
index $i \in B - \{\max(B)\}$ such that $(i, \max(B)) \notin A(a,b)$, then $\pi$ refines the two-block set partition
$\pi' := \{ [i+1, \max(B)], [b-1] - [i+1, \max(B)] \}$ and  $\rot^{-i}(\pi)$  refines
$\rot^{-i}(\pi') = \{ [1, \max(B)-i], [\max(B)-i+1, b-1] \}$.  Since $\rot^{-i}(\pi)$ consists of two intervals, we have that 
$\rot^{-i}(\pi')$ can be obtained from $\rot^{-i}(\pi)$ by a sequence of coarsenings as in Lemma~\ref{valid-rankings-coarsen}.
The condition $(i, \max(B)) \notin A(a,b)$ means that
\begin{equation*}
\left \lceil (\max(B) - i)\frac{a}{b} \right \rceil + 
\left \lceil (b - \max(B) + i - 1)\frac{a}{b} \right \rceil > a,
\end{equation*}
so that $\rot^{-i}(\pi')$ does not have a valid $a,b$-ranking.  By Lemma~\ref{valid-rankings-coarsen}, we conclude that
$\rot^{-i}(\pi)$ does not have a valid $a,b$-ranking. 

By the last paragraph, we may assume that for every block $B \in \krew(\pi)$ and every index $i \in B - \{\max(B)\}$, we have
$(i, \max(B)) \in A(a,b)$.  Since $\pi \in \NC(b-1) - \NC(a,b)$, by Proposition~\ref{kreweras-characterization} there exists a block
$B \in \krew(\pi)$ and indices $i, j \in B - \{ \max(B) \}$ with $i < j$ such that
\begin{equation*}
\left \lceil (k - i)\frac{a}{b} \right \rceil - \left \lceil (k - j)\frac{a}{b} \right \rceil < (j-i)\frac{a}{b},
\end{equation*}
where $\max(B) = k$. 
The set partition $\pi$ refines the three-block noncrossing partition
$\pi' := \{ [i+1, j], [j+1, k], [b-1] - [i+1,k] \}$.  Therefore, the set partition $\rot^{-i}(\pi)$ refines 
$\rot^{-i}(\pi') = \{ [1, j-i], [j-i+1, k-i], [k-i+1, b-1] \}$.  Since $\rot^{-i}(\pi')$ consists of three intervals,
we have that $\rot^{-i}(\pi')$ can be obtained from $\rot^{-i}(\pi)$ by a sequence of coarsenings as in Lemma~\ref{valid-rankings-coarsen}.
We show that $\rot^{-i}(\pi')$ does not have a valid $a,b$-ranking; by Lemma~\ref{valid-rankings-coarsen}, this implies
that $\rot^{-i}(\pi)$ does not have a valid $a,b$-ranking and completes the proof.

Working towards a contradiction, suppose  
$\pi'' := \rot^{-i}(\pi')$ has a valid $a,b$-ranking.   Denote the blocks of $\pi''$ by
$B_1 = [1, j-i], B_2 = [j-i+1, k-i],$ and $B_3 = [k-i+1, b-1]$.  We have that
$\rank^{\pi''}_{a,b}(B_1) + \rank^{\pi''}_{a,b}(B_2) + \rank^{\pi''}_{a,b}(B_3) = a$.
On the other hand, by Lemma~\ref{valid-rankings-coarsen}, we have that $\pi''' := \{B_1 \cup B_2, B_3\}$ also has a valid $a,b$-ranking.
Moreover, we have that $\rank^{\pi'''}_{a,b}(B_3) = \rank^{\pi''}_{a,b}(B_3)$.
This implies that 
\begin{equation*}
a = \rank^{\pi'''}_{a,b}(B_1 \cup B_2) + \rank^{\pi'''}(B_3) = \left \lceil (k-i)\frac{a}{b} \right \rceil + \rank^{\pi''}(B_3).
\end{equation*}
Putting these facts together gives
\begin{align*}
a &= \rank^{\pi''}_{a,b}(B_1) + \rank^{\pi''}_{a,b}(B_2) + \rank^{\pi''}_{a,b}(B_3) \\
&= \left \lceil (j-i)\frac{a}{b} \right \rceil + \left \lceil (k-j)\frac{a}{b} \right \rceil + \rank^{\pi''}(B_3) \\
&= \left \lceil (j-i)\frac{a}{b} \right \rceil + \left \lceil (k-j)\frac{a}{b} \right \rceil  + a - \left \lceil (k-i)\frac{a}{b} \right \rceil \\
&> a + \left \lceil (j-i)\frac{a}{b} \right \rceil  - (j-i)\frac{a}{b} \\
&> a,
\end{align*}
which is a contradiction.  We conclude that $\pi''$ does not have a valid $a,b$-ranking.
\end{proof}

As an application of Proposition~\ref{rank-sequence-characterization}, we get that 
the set $\NC(a,b)$ carries an action of not just the rotation operator $\rot$, but the full dihedral group
$\langle \rot, \rfn \rangle$.  This gives another application of $a,b$-ranks.

\begin{corollary}
\label{dihedral-closed}
Let $a < b$ be coprime.  The set $\NC(a,b)$ of $a,b$-noncrossing partitions is closed under the action of the 
dihedral group $\langle \rot, \rfn \rangle$.
\end{corollary}

\begin{proof}
It is enough to show that $\NC(a,b)$ is closed under $\rfn$.  For any noncrossing partition
$\pi \in \NC(b-1)$ and any block $B$ of $\pi$, we have that
\begin{equation*}
\rank^{\pi}_{a,b}(B) = \rank^{\rfn(\pi)}_{a,b}(\rfn(B)).
\end{equation*}
It follows that every element in the rotation orbit of $\pi$ has a valid $a,b$-ranking if and only if 
every element in the rotation orbit of $\rfn(\pi)$ has a valid $a,b$-ranking.
\end{proof}

\section{Modified rank sequences}
\label{Modified rank sequences}

In this section we will study rational noncrossing partitions which have  nontrivial rotational symmetry. 
Our key tool will the the theory of $a,b$-rank developed in Section~\ref{Characterizations}.
We fix the following

{\bf Notation.} {\em For the remainder of this section, let $a < b$ be  coprime positive integers.  Let $d | (b-1)$ be a divisor with 
$1 \leq d < b-1$.  Let $q := \frac{b-1}{d}$.  Let 
\begin{equation*}
\NC_d(a,b) := \{ \pi \in \NC(a,b) \,:\, \rot^d(\pi) = \pi \} 
\end{equation*}
denote the set of $a,b$-noncrossing
partitions which are fixed by $\rot^d$.}

The numerology of $\NC_d(a,b)$ will turn out to be somewhat simpler than that of $\NC(a,b)$ itself.  We begin by 
defining a modified version of the rank sequence which is well suited to studying partitions
with rotational symmetry.

Let $\pi \in \NC_d(a,b)$.  
The fact that $\pi$ is noncrossing implies that $\pi$ has at most one block $B_0$ which satisfies 
$\rot^d(B_0) = B_0$.  If  such a block $B_0$ exists, it is called {\em central}.
Moreover, the cyclic group $\ZZ_q = \langle \rot^d \rangle$ acts 
freely on the non-central blocks of $\pi$.
A non-central block $B$ of $\pi$ is called {\em wrapping}
if the interval $[\min(B), \max(B)]$ contains every block in the $\langle \rot^d \rangle$-orbit of $B$.
Any $\langle \rot^d \rangle$-orbit of blocks has at most one wrapping block.

\begin{defn}
Let $\pi \in \NC_d(a,b)$.  The {\em $d$-modified rank sequence} of $\pi$ is the length $d$
 sequence $S_d(\pi) = (s_1,  \dots, s_d)$ of nonnegative 
integers defined by 
\begin{equation}
s_i := 
\begin{cases}
\rank^{\pi}_{a,b}(B) & \text{if $i = \min(B)$ for some non-central, non-wrapping block $B \in \pi$,} \\
0 & \text{otherwise.}
\end{cases}
\end{equation}
\end{defn}

For example, suppose that $(a,b) = (4,9)$, $d = 4$, and 
$\pi = \{ \{1,8\}, \{2,3,6,7\}, \{4,5\} \} \in \NC_4(4,9)$.  The block $\{1,8\}$ of $\pi$ is wrapping
and the block $\{2,3,6,7\}$ of $\pi$ is central.  Since the $4,9$-rank of $\{4,5\}$ is $1$, the 
$4$-modified rank sequence of $\pi$ is $S_4(\pi) = (0,0,0,1)$.

The $d$-modified rank sequence $S_d(\pi)$ is like  the ordinary rank sequence $R(\pi)$, but we only consider the indices
in $[d]$ rather than in $[b-1]$ and we only keep track of the ranks of blocks which are neither central nor wrapping.
It is true, but not obvious at this point, that a set partition $\pi \in \NC_d(a,b)$ is determined by $S_d(\pi)$.

Our first lemma states that the assignment $\pi \mapsto S_d(\pi)$ commutes with the action of rotation.

\begin{lemma}
\label{modified-rank-sequences-rotate}
Let $\pi \in \NC_d(a,b)$ and let
$S_d(\pi) = (s_1, s_2, \dots, s_d)$ be the $d$-modified rank sequence of $\pi$.  We have that 
\begin{equation}
S_d(\rot(\pi)) = \rot(S_d(\pi)),
\end{equation}
where $\rot(s_1, s_2, \dots, s_d) = (s_d, s_1, \dots, s_{d-1})$ is the rotation operator on sequences.
\end{lemma}

\begin{proof}
Let $S_d(\rot(\pi)) = (s_1', s_2' ,\dots, s_d')$ be the $d$-modified rank sequence of $\rot(\pi)$ and let 
$1 \leq i \leq q$.  We show that $s'_{i} = s_{i-1}$, where subscripts are interpreted modulo $d$.
We will make free use of Proposition~\ref{dihedral-invariance}, which implies that
$\rank^{\pi}_{a,b}(B) = \rank^{\rot(\pi)}_{a,b}(\rot(B))$ for any block $B \in \pi$.  

{\bf Case 1:}  {\em $2 \leq i \leq d$.}  Suppose $s_{i-1} > 0$.  Then $i - 1 = \min(B)$ for some non-central, non-wrapping
block $B \in \pi$.  Since $B$ is non-central and non-wrapping and $1 \leq \min(B) \leq d-1$, we know that $\rot(B)$ is also
non-central and non-wrapping with $\min(\rot(B)) = i$.  We conclude that $s_i' = s_{i-1}$.

Suppose $s_{i - 1} = 0$.  If $i-1$ is not the minimum element of a block of $\pi$, then $i$ is not the minimum element of a block of 
$\rot(\pi)$, so that $s_i' = 0$.  If $i-1 = \min(B_0)$ for a central block $B_0 \in \pi$, then $\rot(B_0)$ is a central block in $\rot(\pi)$
with $i \in \rot(B_0)$, so that $s_i' = 0$.  
If $i-1 = \min(B)$ for a wrapping block $B \in \pi$, then the fact that $1 \leq \min(B) \leq d-1$ implies that 
either $i \neq \min(\rot(B))$ or 
($\rot(B)$ is wrapping with $i \in \rot(B)$).  In either situation, we get that $s'_i = 0$.

{\bf Case 2:}  {\em $i = 1$.}
Suppose $s_d > 0$.  Then $d = \min(B)$ for some non-central, non-wrapping block $B \in \pi$.
Recalling that $\rot^d(\pi) = \pi$, 
it follows that $\rot^{d(q-1)+1}(B)$ is a non-central, non-wrapping block of $\rot(\pi)$ containing $1$.
Thus, we get $s'_1 = \rank^{\rot(\pi)}_{a,b}(\rot^{d(q-1)+1}(B)) = \rank^{\pi}_{a,b}(B) = s_q$.

Suppose $s_d = 0$.  If $d$ is contained in a central block of $\pi$, then $1$ is contained in a central
block of $\rot(\pi)$ and $s'_1 = 0$.  Since $\pi$ is noncrossing and $\rot^d(\pi) = \pi$, the index $q$
cannot be contained in a wrapping block of $\pi$.  If $d \in B$ for some block $B \in \pi$ which is non-central
and non-wrapping, we must have that $d \neq \min(B)$.  Since $\pi$ is noncrossing with $\rot^d(\pi) = \pi$, 
it follows that $\rot(B)$ is wrapping and $1 \in \rot(B)$, so that $s'_1 = 0$.
\end{proof}

What sequences $(s_1,  \dots, s_d)$ of nonnegative integers arise as $d$-modified rank sequences 
of partitions in $\NC_d(a,b)$?
If $\pi \in \NC_d(a,b)$ and $S_d(\pi) = (s_1,  \dots, s_d)$ is the $d$-modified rank sequence of $\pi$,
we claim that 
\begin{equation*}
q(s_1 + \cdots + s_d) = \sum_B \rank^{\pi}_{a,b}(B),
\end{equation*}
where the sum is over all non-central blocks $B \in \pi$.  Indeed, each $q$-element orbit of non-central blocks contributes 
the rank of one of its constituents precisely once to the nonzero terms in $S_d(\pi)$.  
By Proposition~\ref{rank-sequence-characterization},
\begin{equation*}
s_1 +  \cdots + s_d \leq \frac{a}{q},
\end{equation*}
with equality if and only if $\pi$ does not have a central block.   
(Unless $q | a$, the partition $\pi$  necessarily has a central block.)

We call a length $d$ sequence $(s_1, \dots, s_d)$ of nonnegative integers {\em good}
if we have the inequality
$s_1 +  \cdots + s_d \leq \frac{a}{q}$.
The goal for the remainder of this section is to show that the map
\begin{equation*}
S_d: \NC_d(a,b) \longrightarrow \left\{ \text{good sequences $(s_1, \dots, s_d)$} \right\}
\end{equation*}
is a bijection.
Since good sequences are easily enumerated, this will give us  information about
$\NC_d(a,b)$.
The strategy is to isolate nice subsets of $\NC_d(a,b)$ and the set of good sequences which contain at least one 
representative from every rotation orbit, show that these subsets are in bijection under $S_d$,
and apply Lemma~\ref{modified-rank-sequences-rotate}.

\begin{defn}
\label{def:noble-partition}
A set partition $\pi \in \NC_d(a,b)$ is {\em noble} if $\pi$ does not contain any wrapping blocks and,
if $\pi$ contains a central block $B_0$, we have that $1 \in B_0$.
\end{defn}

For example, consider the case $(a,b) = (6,7)$ and $d = 3$.  We have that 
$\pi = \{ \{1\}, \{2,3,5,6\}, \{4\} \} \in \NC_3(6,7)$ is not a noble partition because $1$ is not
contained in the central block.  On the other hand,
both of the rotations 
$\{ \{1,3,4,6\}, \{2\}, \{5\} \}$ and $\{ \{1,2,4,5\}, \{3\}, \{6\} \}$ of $\pi$ are noble partitions.
Also, if $(a,b) = (6,7)$ and $d = 2$ we have that 
$\sigma = \{ \{1,6\}, \{2,3\}, \{4,5\} \} \in \NC_2(6,7)$ is not a noble partition because the block
$\{1,6\}$ is wrapping.  However, the rotation
$\{ \{1,2\}, \{3,4\}, \{5,6\} \}$ of $\sigma$ is a noble partition.  
In general, we have the following observation.

\begin{observation}
\label{noble-partitions-cycle}
Every $\langle \rot \rangle$-orbit in $\NC_d(a,b)$ contains at least one noble partition. 
\end{observation}

The  notion of nobility for sequences involves an intermediate definition.
Let $s  = (s_1, \dots, s_d)$ be a good sequence.
We call $s$ {\em very good} if 
$s_1 + \cdots + s_d = \frac{a}{q}$ or $s_1 = 0$.

We define a map
\begin{equation*}
L: \{ \text{very good sequences} \} \longrightarrow \{ \text{lattice paths from $(0, 0)$ to $(b,a)$} \}
\end{equation*}
as follows.  
If $s = (s_1, \dots, s_q)$ is a very good sequence, let $L(s)$ be the lattice path obtained by taking the $q$-fold
concatenation of
the path $N^{s_1} E \dots N^{s_d} E$, adding a terminal east step, and if $s_1 + \cdots + s_d < \frac{a}{q}$ 
adding an initial vertical run of size $c := a - q(s_1 + \cdots + s_d)$.  In symbols,
\begin{equation*}
L(s) := \begin{cases}
(N^{s_1} E \dots N^{s_d} E)^q E & \text{if $s_1 + \cdots + s_d = \frac{a}{q}$,} \\
(N^c E N^{s_2} E \dots N^{s_d} E) (N^{s_1} E \dots N^{s_d} E)^{q-1} E & \text{if $s_1 + \cdots + s_d < \frac{a}{q}$.}
\end{cases}
\end{equation*}
Since $s$ is assumed to be very good, we get that $L(s)$ ends at $(b, a)$ so that the map $L$ is well defined.  
We will refer to the subpaths $L_1, \dots, L_q$ defined by the above factorization of $L(s)$ as the {\em segments}
of $L(s)$, so that $L(s) = L_1 \cdots L_q E$.

If $s$ is a very good sequence,
the lattice path $L(s)$ is typically {\em not}
an $a,b$-Dyck path.  For example, if $(a,b) = (4,7)$, $d = 3$, and $s = (s_1, s_2, s_3) = (0,1,1)$, we have that 
\begin{equation*}
L(s) = (N^0 E N^1 E N^1 E) (N^0 E N^1 E N^1 E) E,
\end{equation*}
which is not a $4,7$-Dyck path.

\begin{defn}
\label{def:noble-sequence}
A very good sequence $s = (s_1, \dots, s_d)$ is {\em noble} if $L(s)$ is an $a,b$-Dyck path.
\end{defn}

When $(a,b)  = (4,7)$ and $d = 3$, the above example shows that $(0, 1, 1)$ is not a noble sequence.
On the other hand, the rotation $(1,1,0)$ of $(0,1,1)$ is a noble sequence.
In general, we have the following analog of Observation~\ref{noble-partitions-cycle} for good sequences.

\begin{lemma}
\label{noble-sequences-cycle}
Every good sequence is $\langle \rot \rangle$-conjugate to at least one noble sequence.
\end{lemma}

\begin{proof}
Let $(s_1, \dots, s_d)$ be a good sequence and set
$c = a - q(s_1 + \cdots + s_d)$.

{\bf Case 1:}  {\em $c = a$.}
In this case $(s_1, \dots, s_d)$ is  the zero sequence $(0, 0, \dots, 0)$ and is trivially noble.

{\bf Case 2:} {\em $0 < c < a$.}  The argument we present here is a modification
of the argument used to prove the Cycle Lemma. 

Consider the lattice path $L$ which
starts at the origin and ends at $(2d, 2(s_1 + \cdots + s_d))$ given by 
\begin{equation*}
L = (N^{s_1} E \dots N^{s_d} E)(N^{s_1} E \dots N^{s_d} E). 
\end{equation*} 
As is common in rational Catalan theory,
we label the lattice points $P$ on $L$ with integers $w(P)$ as follows.  
The origin is labeled $0$.
Reading $L$ from left to right, if $P$ and $P'$ are consecutive lattice points, we set 
$w(P') = w(P) - a$ if $P'$ is connected to $P$ with an $E$-step and
$w(P') = w(P) + b$ if $P'$ is connected to $P$ with an $N$-step.

For example, suppose that $(a,b) = (11,13)$, $d = 4$, and $(s_1, s_2, s_3, s_4) = (1,0,2,0)$.  
We have that $q = \frac{13}{4} = 3$ and $c = 11 - 3(1+0+2+0) = 2$.
The lattice path $L$, together with the labels of its lattice points, is as follows.

\begin{center}
\includegraphics[scale = 0.4]{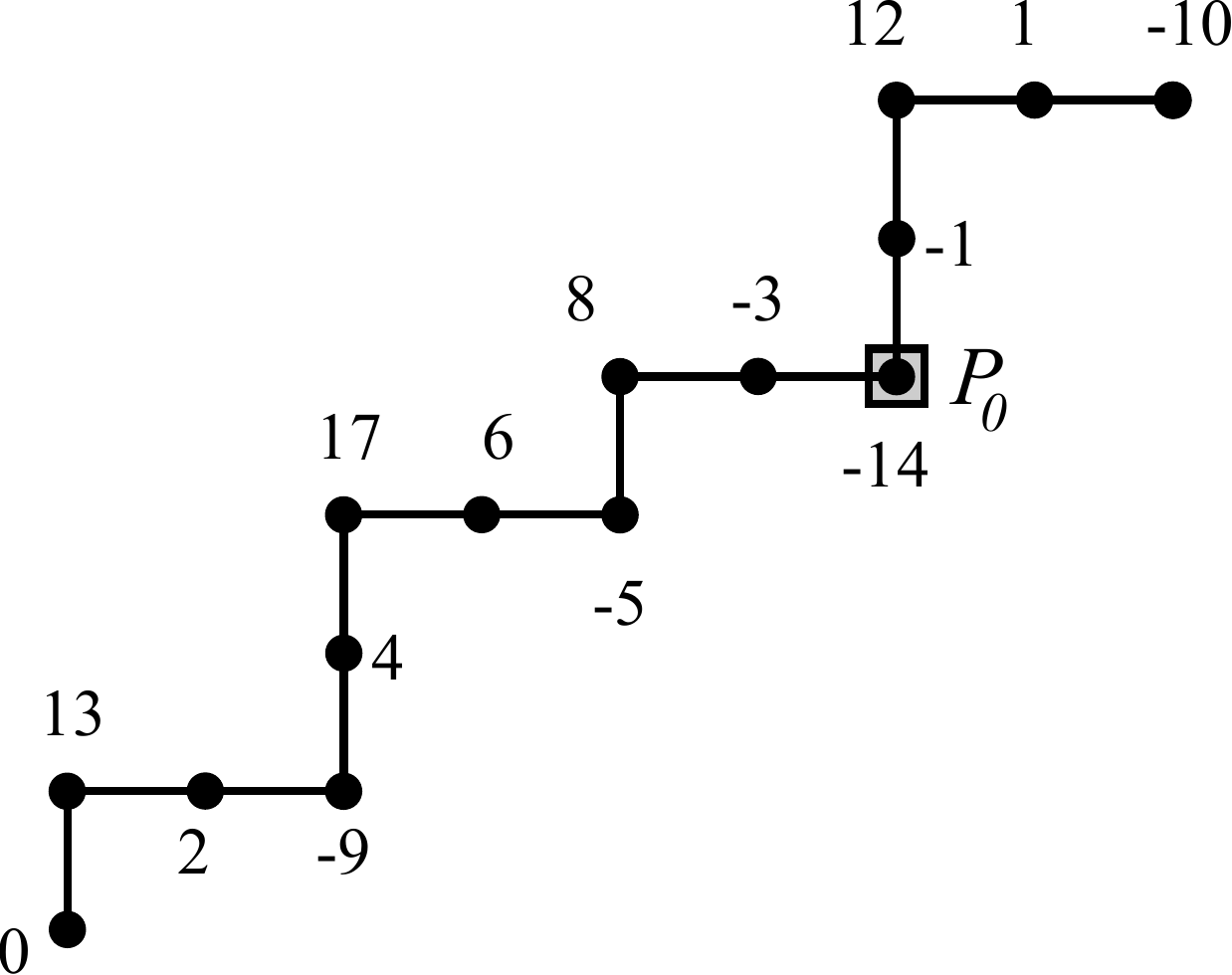}
\end{center}

By coprimality, there exists a unique lattice point on $L$ of minimal weight.
Let $P_0$ be this lattice point; in the above example, we have $w(P_0) = -14$.
We claim that $P_0$ occurs after a 
pair of consecutive east steps $EE$.  Indeed, since $0 < c < a$, we know that the weight of the terminal
lattice point of $L$ is negative (in the above example, $-10$), so that $P_0$ is not the origin. 
If $P_0$ did not 
occur after a pair of consecutive east steps, by the
minimality of $w(P_0)$, we have that $P_0$ occurs after a pair $NE$.  But since $a < b$
the lattice point $P'_0$ occurring at the beginning of this $NE$-sequence satisfies 
$w(P'_0) = w(P_0) + a - b < w(P_0)$, a contradiction.  Therefore, the lattice point $P_0$ does 
indeed occur after a pair of consecutive east steps $EE$.

The minimality of $w(P_0)$ and the fact
 that $(s_1, \dots, s_d) \neq (0, \dots, 0)$ mean that $P_0$ is immediately followed by a nonempty vertical
run $N^{s_i}$ for some $1 \leq i \leq d$.  Since $P_0$ is preceded by a pair $EE$, we get that $i \geq 2$
and $s_{i-1} = 0$.  
In the above example, we have that $i = 3$.
Since $s_{i-1} = 0$, the rotated sequence 
$(s_{i-1}, s_i, \dots, s_d, s_1, s_2, \dots, s_{i-2})$ is very good.  
The lattice path $L(s_{i-1}, s_i, \dots, s_d, s_1, s_2, \dots, s_{i-2})$ is therefore well defined.

In the above example,  we have $(s_{i-1}, s_i, \dots, s_d, s_1, s_2, \dots, s_{i-2}) = (0,2,0,1)$
and the lattice path $L(0,2,0,1)$ is shown in the proof of Lemma~\ref{surjective} below.  The segmentation
$L(s) = L_1 \dots L_q E = L_1 L_2 L_3 E$ is shown with vertical hash marks.  Observe that $L(s)$
is an $a,b$-Dyck path in this case.

We claim that the lattice path 
$L(s_{i-1}, s_i, \dots, s_d, s_1, s_2, \dots, s_{i-2})$ is always an $a,b$-Dyck path, so that
the rotation 
$(s_{i-1}, s_i, \dots, s_d, s_1, s_2, \dots, s_{i-2})$ is a noble sequence. 
Indeed, consider the segmentation $L(s) = L_1 \dots L_q E$.  The segments $L_1, \dots, L_q$ will be progressively
further east, so it is enough to show that the final segment $L_q$ stays west of the line $y = x$.
By construction, the segment $L_q$ starts with a single east step, hits the copy of the point $P_0$, then
then has a nonempty vertical run, and eventually ends at the point $(b-1,a)$.  Since $(s_1, \dots, s_d)$ 
is a good sequence, we know that $L_q$ starts at a lattice point to the west of the line 
$y = \frac{a}{b}(x+1)$.  The minimality of $w(P_0)$ forces $L_q$ to remain west of the line 
$y = \frac{a}{b} x$.

{\bf Case 3:} {\em $c = 0$.} This is a special case of the Cycle Lemma; the argument is very similar to Case 2
and is only sketched.

We again consider the lattice path $L$  given by 
$L = (N^{s_1} E \dots N^{s_d} E)(N^{s_1} E \dots N^{s_d} E)$ and assign weights $w(P)$ to the lattice points 
$P$ on $L$ as before.  There exists a unique lattice point $P_0$ on $L$ with minimal weight.  
The lattice point $P_0$ necessarily occurs before a nonempty vertical run $N^{s_i}$ for some $1 \leq i \leq d$.
The rotation $(s_i, s_{i+1}, \dots, s_d, s_1, s_2, \dots, s_{i-1})$ of $s$ is a noble sequence.
\end{proof}

Given any noble sequence $s$, we may consider the $a,b$-noncrossing partition $\pi(L(s))$ corresponding
to the Dyck path $L(s)$.  We prove that $\pi(L(s))$ is $\rot^d$-invariant and noble.

\begin{lemma}
\label{surjective}
Suppose that $s = (s_1, \dots, s_d)$ is a noble sequence
Then $\pi := \pi(L(s))$ is a noble $a,b$-noncrossing partition  (and in particular $\rot^d(\pi) = \pi$) with $S_d(\pi) = s$.
\end{lemma}

\begin{proof}
Recall the visibility bijection between blocks of $\pi$ and 
nonempty vertical runs in $L(s)$.  
Let $c = a - q(s_1 + \cdots + s_d)$.
The argument depends on whether $c > 0$ or $c = 0$.

{\bf Case 1:} $c > 0$.
We consider the segmentation of $L(s)$ given by
\begin{equation*}
L(s) = L_1 L_2 \dots L_q E,
\end{equation*}
 where
$L_1 = N^c E N^{s_2} E \dots N^{s_d} E$ and
$L_j = N^{s_1} E N^{s_2} E \dots N^{s_d} E$ for $2 \leq j \leq q$.

As an example of this case, consider $(a, b) = (11, 13)$, $d = 4$, $s = (0,2,0,1)$.
Then $s$ is a noble sequence.
The $11,13$-Dyck path 
$L(s) = L_1 L_2 L_3 E$ is shown below, where vertical hash marks separate the segments
$L_1, L_2,$ and $L_3$.  The valley lasers from $L(s)$ are as shown.  For visibility,
we suppress the labels on the interior lattice points.

\begin{center}
\includegraphics[scale = 0.3]{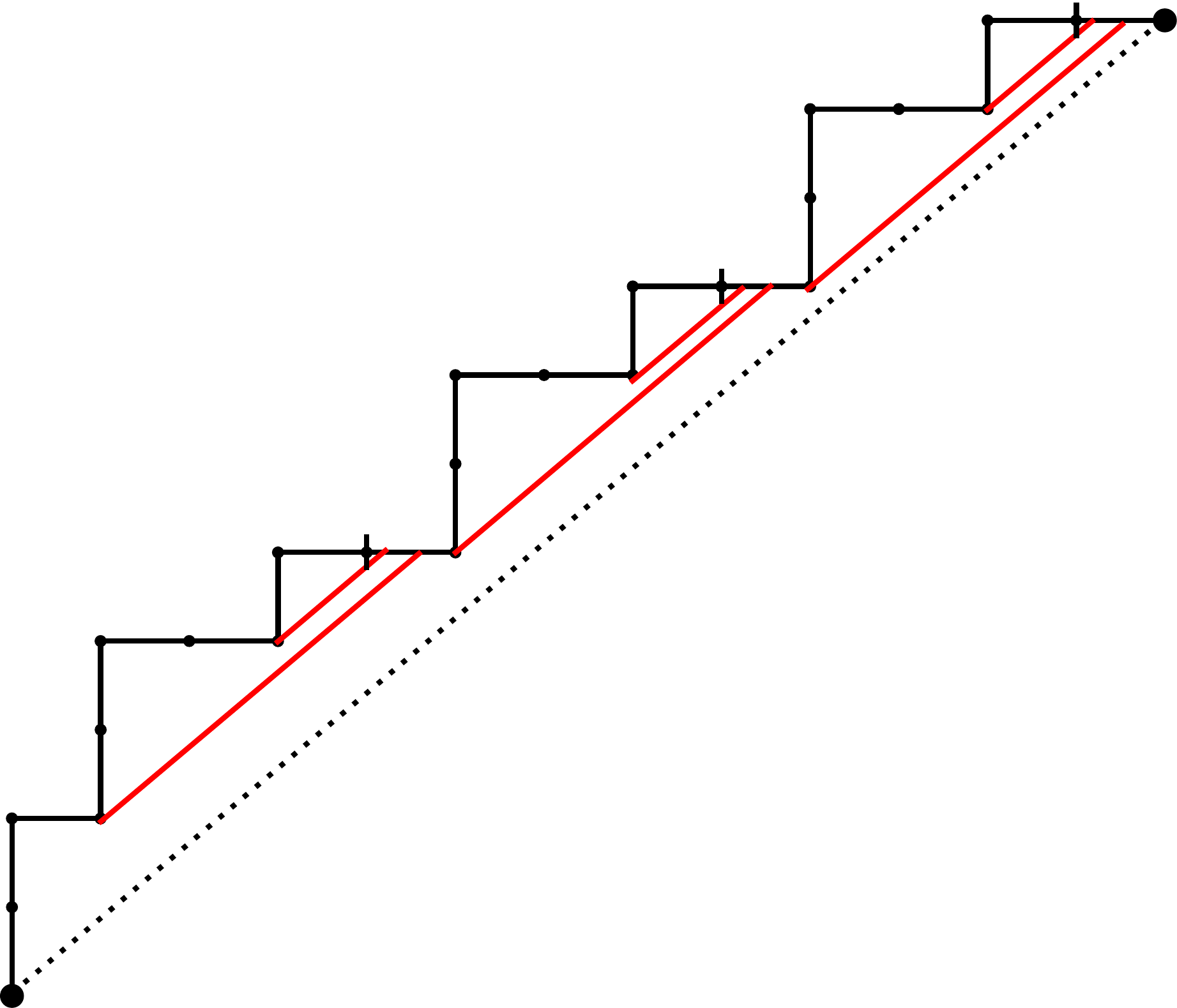}
\end{center}

Since $s$ is noble and $c > 0$, we have $s_1 = 0$.
Fix any index $2 \leq  i \leq d$ such that $s_i > 0$ and another index $1 \leq j \leq q-1$.
Both of the segments $L_j$ and $L_{j+1}$ of the $a,b$-Dyck path $L(s)$ contain a copy of the nonempty vertical
run $N^{s_i}$.
If $P_0$ and $P_1$ are the valleys at the bottom of these vertical runs in $L_j$ and $L_{j+1}$, respectively, 
we have that the lasers $\ell(P_0)$ and $\ell(P_1)$ fired from $P_0$ and $P_1$ are (rigid)
translations of 
the same line segment.  
In particular, the block visible from the copy of $N^{s_i}$ in $L_{j+1}$
is the image of the block visible from the copy of $N^{s_i}$ in $L_j$ by the operator $\rot^d$.  Moreover,
the fact that $s_1 = 0$ means that neither of these blocks contain the index $1$.

In the above example, we have three segments ($L_1, L_2,$ and $L_3$) with $s_i > 0$ for $i = 2, 4$.
The blocks visible from the copies of $N^{s_2} = N^2$ in $L_1, L_2,$ and $L_3$ are
$\{2,3\}, \{6,7\},$ and $\{10,11\}$, respectively.  The blocks visible from the copies of $N^{s_4} = N^1$
in $L_1, L_2,$ and $L_3$ are $\{4\}, \{8\},$ and $\{12\}$, respectively.  The block visible from
the initial vertical run is $\{1,5,9\}$, which is central.
 
In general, we conclude that 
the set of blocks of $\pi$ not containing $1$ is stable under the action of $\rot^d$, so that the block
of $\pi$ containing $1$ must be central. 
 We get that $\pi \in \NC_d(a,b)$.  Since $\pi \in \NC_d(a,b)$ has a central block
containing $1$, the partition $\pi$ contains no wrapping blocks.  It follows that  $\pi$ in noble and
$S_d(\pi) = s$.

{\bf Case 2:}  $c = 0$.
In this case,
$\pi$ does not contain a central block.  
We again consider the segmentation $L(s) = L_1 L_2 \dots L_q E$ as in Case 1.
Here $L_j = N^{s_1} E N^{s_2} E \dots N^{s_d} E$ for $1 \leq j \leq q$.

As an example of this case, consider $(a,b) = (9,13), d = 4,$ and $s = (1,2,0,0)$. 
We have that $s$ is a noble sequence.
 The $9,13$-Dyck
path $L(s)$ is shown below, with diagonal hash marks denoting the segmentation
$L(s) = L_1 L_2 L_3 E$.  The valley lasers of $L(s)$ are shown, and the interior lattice point
labels are suppressed.

\begin{center}
\includegraphics[scale = 0.3]{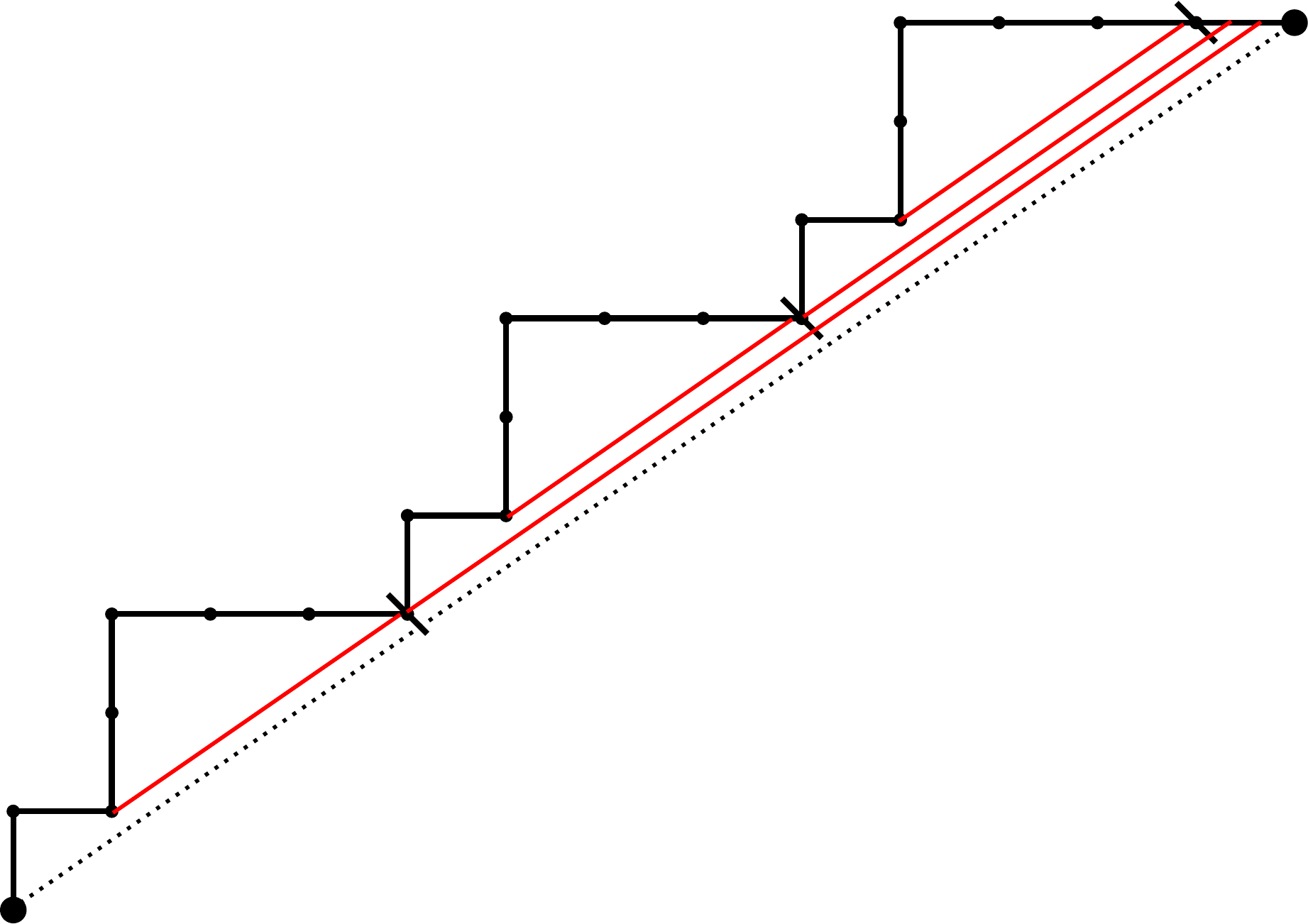}
\end{center}

Let $1 \leq i \leq d$ be an index such that $s_i > 0$ and consider any two 
consecutive segments
$L_j$ and $L_{j+1}$ of $L(s)$.  
The lasers fired from the valleys below the copies of $N^{s_i}$ both $L_j$ and $L_{j+1}$ are either
\begin{itemize}
\item  translates of each other, or
\item both hit the Dyck path $L(s) = L_1 \dots L_q E$ on its terminal east step.
\end{itemize}
In the above example, we see that the lasers fired from each copy of $N^{s_2} = N^2$ are translates of one another,
and the lasers fired from each copy of $N^{s_1} = N^1$ hit $L(s)$ on its terminal east step.
This implies that the blocks corresponding of these vertical runs are $\langle \rot^d \rangle$-conjugate, so that 
$\pi \in \NC_d(a,b)$.  

Moreover, since $c = 0$ the laser fired from any copy of  the (nonempty) vertical run $N^{s_1}$ in 
$L(s) = L_1 \dots L_q E$ must hit $L(s)$ on its terminal east step.  This implies that $\pi$ has no wrapping blocks, so that 
$\pi$ is noble and $S_d(\pi) = s$.  
\end{proof}

We  show that nobility for sequences and nobility for partitions coincide.

\begin{lemma}
\label{injective}
Suppose that $\pi \in \NC_d(a,b)$. Then $\pi$ is noble if and only if $S_d(\pi)$ is noble.
\end{lemma}

\begin{proof}  Write $S_d(\pi) =  (s_1, \dots, s_d)$ and set $c = a - q(s_1 + \cdots + s_d)$.

($\Rightarrow$)  Suppose $\pi$ is noble. 
 Since $\pi$ contains no wrapping blocks, the rank sequence $R(\pi)$ is given by
the formula 
\begin{equation*}
R(\pi) = 
\begin{cases}
(s_1, \dots, s_d, s_1, \dots, s_d, \dots, s_1, \dots, s_d) & \text{if $s_1 + \cdots + s_d = \frac{a}{q},$} \\
(c, \dots, s_d, s_1, \dots, s_d, \dots, s_1, \dots, s_d) & \text{if $s_1+ \cdots + s_d < \frac{a}{q}.$}
\end{cases}
\end{equation*}
If $c > 0$, since $1$ is contained in the central block of $\pi$ we get $s_1 = 0$, so that $S_d(\pi)$ is very good.
By Proposition~\ref{recover-dyck-path}, we get that $L(s)$ is an $a,b$-Dyck path, so that $S_d(\pi)$ is noble.

($\Leftarrow$)  Suppose $S_d(\pi)$ is noble.  We claim that $\pi$ contains no wrapping blocks.  
Working towards a contradiction, assume that $\pi$ contains at least one wrapping block and choose a wrapping
block $B$ of $\pi$ such that the interval $[\min(B), \max(B)]$ is maximal under containment.

If $1 \in B$, then we would have $s_1 = 0$ since $B$ is wrapping, making the first step of
the path $L(S_d(\pi))$ an east step.  This contradicts the nobility of $S_d(\pi)$.  We conclude that $1 \notin B$.

Since $1 \notin B$, the interval $[1, \min(B) - 1]$ is nonempty.
By our choice of $B$,
we have that $[1, \min(B) - 1]$ is a union
blocks of $\pi$, none of which are central or wrapping.
This means that the first $\min(B) - 1$ terms in the $d$-modified rank sequence
$S_d(\pi)$ coincide with the first $\min(B) - 1$ terms in the ordinary rank sequence
$R(\pi)$.

By Proposition~\ref{recover-dyck-path},
the $a,b$-Dyck path corresponding to $\pi$  starts at the origin
 with the subpath $N^{s_1} E N^{s_2} E \dots N^{s_{\min(B)-1}} E$.
 Since $[1, \min(B) - 1]$ is a union of blocks of $\pi$, 
 Corollary~\ref{ranks-add} guarantees that
 this subpath
ends at the point $(\min(B) - 1, \left \lceil \frac{a}{b} (\min(B) - 1) \right \rceil)$.
It follows that the subpath $(N^{s_1} E N^{s_2} E \dots N^{s_{\min(B)-1}} E)E$
 obtained by appending a single east step
crosses the diagonal $y = \frac{a}{b} x$.  But since
$B$ is wrapping, we have $s_{\min(B)} = 0$, so that 
this is an initial subpath of $L(S_d(\pi))$, so that
$L(S_d(\pi))$ is not an $a,b$-Dyck path. 
This contradicts the nobility of $S_d(\pi)$.  We conclude that $\pi$ contains no wrapping blocks.

Suppose that $\pi$ contains a central block $B_0$.  We need to prove that $1 \in B_0$.  
If $1 \notin B_0$, the fact that $\pi$ does not contain wrapping blocks implies that 
$[1, \min(B_0) - 1]$ is a union of non-wrapping, non-central blocks of $\pi$.  
The first $\min(B_0) - 1$ terms of the $d$-modifed rank sequence $S_d(\pi)$ coincide with the 
corresponding terms of the ordinary rank sequence $R(\pi)$.  The same reasoning as
in the last paragraph implies that the lattice path
$L(S_d(\pi))$ contains the point 
$(\min(B_0) - 1, \left \lceil \frac{a}{b} (\min(B_0) - 1) \right \rceil)$.  However, since $B_0$ is central,
we have that $s_{\min(B_0)} = 0$, so that the lattice path $L(S_d(\pi))$ has an east step originating 
from this lattice point.  But this means that $L(S_d(\pi))$ is not an $a,b$-Dyck path, contradicting
the nobility of $S_d(\pi)$.  We conclude that $1 \in B_0$, and that $\pi$ is a noble partition.
\end{proof}

We have the lemmata we need to prove that the map $S_d$ is a bijection.  In particular, partitions in
$\NC_d(a,b)$ are determined by their $d$-modified rank sequences.

\begin{proposition}
\label{modified-rank-sequence-unique}
The map 
$S_d: \NC_d(a,b) \longrightarrow \left\{ \text{good sequences $(s_1, \dots, s_d)$} \right\}$
is a bijection which commutes with the action of rotation.
\end{proposition}

\begin{proof}
By Lemma~\ref{modified-rank-sequences-rotate}, we know that $S_d$ commutes with rotation.  If $s$ is a noble sequence,
Lemma~\ref{surjective} shows that $S_d^{-1}(s)$ is nonempty, and Lemma~\ref{noble-sequences-cycle} shows that
$S_d$ is surjective.

To see that $S_d$ is injective, let $s$ be a noble sequence and suppose $\pi \in \NC_d(a,b)$ satisfies
$S_d(\pi) = s = (s_1, \dots, s_d)$.  By Lemma~\ref{injective}, the partition $\pi$ is noble.  The rank sequence $R(\pi)$ is therefore
\begin{equation*}
R(\pi) = 
\begin{cases}
(s_1, \dots, s_d, s_1, \dots, s_d, \dots, s_1, \dots, s_d) & \text{if $s_1 + \cdots + s_d = \frac{a}{q},$} \\
(c, \dots, s_d, s_1, \dots, s_d, \dots, s_1, \dots, s_d) & \text{if $s_1+ \cdots + s_d < \frac{a}{q},$}
\end{cases}
\end{equation*}
where $c = a - q(s_1 + \cdots + s_d)$.
By Proposition~\ref{recover-dyck-path},
any $a,b$-noncrossing partition is determined by its rank sequence. 
We conclude that $|S_d^{-1}(s)| = 1$.
By Observation~\ref{noble-partitions-cycle} and Lemma~\ref{noble-sequences-cycle}, together with the 
fact that $S_d$ commutes with $\rot$ (Lemma~\ref{modified-rank-sequences-rotate}),
we conclude that $S_d$ is a bijection.
\end{proof}

Since every nonzero entry in a $d$-modified rank sequence $S_d(\pi) = (s_1, \dots, s_d)$ corresponds to a
$\langle \rot^d \rangle$-orbit
of non-central blocks of $\pi$ of that rank, the following result follows immediately from 
Proposition~\ref{modified-rank-sequence-unique}.

\begin{corollary}
\label{symmetric-kreweras-count}
Let $a < b$ be coprime positive integers and let $d | (b-1)$ be a divisor with $1 \leq d < b-1$.
Let $m_1, \dots, m_a$ be nonnegative integers which satisfy $\frac{b-1}{d}(m_1 + 2m_2 \cdots + am_a) \leq a$.
Write $m := m_1 + m_2 + \cdots + m_a$.

The number of $a,b$-noncrossing partitions which are invariant under $\rot^d$ and have $m_i$ 
orbits of non-central blocks 
of $a,b$-rank $i$ under the action of $\langle \rot^d \rangle$ is the multinomial coefficient
\begin{equation*}
{d \choose m_1, m_2, \dots, m_a, d-m}.
\end{equation*}
\end{corollary}

The Fuss-Catalan case $b \equiv 1$ (mod $a$) of 
Corollary~\ref{symmetric-kreweras-count} is equivalent to a result of Athanasiadis
\cite[Theorem 2.3]{Ath}. 

We can also consider counting partitions in $\NC_d(a,b)$ with a fixed number of non-central block orbits 
(of any rank).
By Proposition~\ref{modified-rank-sequence-unique}, this is equivalent to counting sequences 
$(s_1, \dots, s_d)$ of nonnegative integers with bounded sum and a fixed number of nonzero entries.

\begin{corollary}
\label{symmetric-narayana-count}
Let $a < b$ be coprime positive integers and let $d | (b-1)$ be a divisor with $1 \leq d < b-1$.
Let $p$ be a nonnegative integer with $\frac{b-1}{d} p \leq a$.

The number of $a,b$-noncrossing partitions which are invariant under $\rot^d$, have a central block,  
and have $p$
orbits of non-central blocks under the action of $\langle \rot^d \rangle$ is
\begin{equation*}
{d \choose p} {  \lfloor \frac{ad}{b-1}  \rfloor - 1 \choose p}.
\end{equation*}

The number of $a,b$-noncrossing partitions which are invariant under $\rot^d$, do not have a central block,
and have $p$ orbits of non-central blocks under the action of $\langle \rot^d \rangle$ is 
\begin{equation*}
\begin{cases}
{d \choose p} {   \frac{ad}{b-1} - 1 \choose p - 1} & \text{if $\frac{b-1}{d} \mid a$,} \\
0 & \text{if $\frac{b-1}{d} \nmid a$.}
\end{cases}
\end{equation*}
\end{corollary}

\begin{proof}
For the first part,
we choose $p$ entries in the sequence $(s_1, \dots, s_d)$ to be nonzero.  Then,
we assign positive values to these $p$ entries in such a way that their sum is
$< \left \lfloor \frac{ad}{b-1} \right \rfloor$.  The second part is similar.
\end{proof}

Finally, we can consider the problem of counting $\NC_d(a,b)$ itself.  By 
Proposition~\ref{modified-rank-sequence-unique}, this corresponds to counting  
sequences $(s_1, \dots, s_d)$ of nonnegative integers which satisfy
$ s_1 + \cdots + s_d \leq \left \lfloor \frac{ad}{b-1} \right \rfloor$.

\begin{corollary}
\label{symmetric-catalan-count}
Let $a < b$ be coprime positive integers and let $d | (b-1)$ be a divisor with $1 \leq d < b-1$.

The number of $a,b$-noncrossing partitions which are invariant under $\rot^d$ is
\begin{equation*}
{  \lfloor \frac{ad}{b-1}  \rfloor + d \choose d}.
\end{equation*}
\end{corollary}

The Fuss-Catalan cases $b \equiv 1$ (mod $a$) of Corollaries~\ref{symmetric-narayana-count}  and
\ref{symmetric-catalan-count} are results
of Reiner \cite[Propositions 6 and 7]{Reiner}.

\section{Cyclic sieving}
\label{Cyclic sieving}

Let $X$ be a finite set, let $C = \langle c \rangle$ be a finite cyclic group acting on $X$, let 
$X(q) \in \mathbb{N}[q]$ be a polynomial with nonnegative integer coefficients, and let 
$\zeta \in \CC$ be a root of unity with multiplicative order $|C|$.
The triple $(X, C, X(q))$ {\em exhibits the cyclic sieving phenomenon (CSP)} if, for all $d \geq 0$, 
we have 
$X(\zeta^d) = |X^{c^d}| = | \{x \in X \,:\, c^d.x = x\} |$ (see \cite{RSWCSP}).  In this section we prove cyclic sieving
results for the action of rotation on $\NC(a,b)$.

Our proofs will be `brute force' and use direct root-of-unity evaluations of $q$-analogs.  We will
make frequent use of the following fact:
If $x \equiv y$ (mod $z$), then
\begin{equation*}
\lim_{q \rightarrow e^{2 \pi i/z}} \frac{[x]_q}{[y]_q} = 
\begin{cases}
\frac{x}{y} & \text{if $y \equiv 0$ mod $z$,} \\
0 & \text{otherwise.}
\end{cases}
\end{equation*}
From this, we get the useful fact that
\begin{equation*}
\lim_{q \rightarrow e^{2 \pi i/z}} \frac{[nz]_q!}{[kz]_q!} = {n \choose k} [nz - kz]_q! |_{q = e^{2 \pi i/z}}.
\end{equation*}

\begin{theorem}
\label{kreweras-csp}
Let $a < b$ be coprime and let ${\bf r} = (r_1, r_2, \dots, r_a)$ be a sequence of nonnegative integers satisfying 
$\sum_{i = 1}^a i r_i = a$.  
Set $k := \sum_{i = 1}^a r_i$.
Let $X$ be the set of $a,b$-noncrossing partitions of $[b-1]$ with $r_1$ blocks of rank $1$,
$r_2$ blocks of rank $2, \dots,$ and $r_a$ blocks of rank $a$.  

The triple $(X, C, X(q))$ exhibits the cyclic sieving phenomenon, where $C = \ZZ_{b-1}$ acts on $X$ by rotation and
\begin{equation}
X(q) = \Krew_q(a,b, {\bf r}) = \frac{[b-1]!_q}{[r_1]!_q [r_2]!_q \cdots [r_a]!_q [b-k]!_q}
\end{equation}
is the {\em $q$-rational Kreweras number}.
\end{theorem}

\begin{proof}
Reiner and Sommers proved that
the $q$-Kreweras number $\Krew_q(a,b, {\bf r})$ is polynomial in $q$ with
nonnegative integer coefficients using algebraic techniques \cite{ReinerSommers}.
No combinatorial proof of the polynomiality or the positivity of 
$\Krew_q(a,b, {\bf r})$ is known.

Let $\zeta = e^{\frac{2 \pi i}{b-1}}$ and let $d | (b-1)$ with $1 \leq d < b-1$.  Write $t = \frac{b-1}{d}$.  
We have that 
$X(\zeta^d) = 0$ unless at $t | r_i$ for all but at most one $1 \leq i \leq a$, and that $r_{i_0} \equiv 1$ (mod $t$)
if $t \nmid r_{i_0}$.  If the sequence ${\bf r}$ satisfies the condition of the last sentence, define a new 
sequence $(m_1, \dots, m_a)$ by $m_i = \left \lfloor \frac{r_i}{t} \right \rfloor$ for $1 \leq i \leq a$.   Let 
$m = m_1 + \cdots + m_a$.
Write $r_{i_0} = c_{i_0} t + s_{i_0}$ for $s_{i_0} \in \{0, 1\}$ and assume $t | r_i$ for all $i \neq i_0$.  We have

\begin{align*}
\lim_{q \rightarrow \zeta^d} X(q) &=
{d \choose m_1} {d - m_1 \choose m_2} \cdots {d - (m - m_a) \choose m_a}
\lim_{q \rightarrow \zeta^d} \frac{[b-1-mt]!_q [r_{i_0} - s_{i_0}]!_q}{[r_{i_0}]!_q [b-k]!_q} \\
&= {d \choose m_1, \dots, m_a, d-m} 
\lim_{q \rightarrow \zeta^d} \frac{[b-1-(k - s_{i_0})]!_q [r_{i_0} - s_{i_0}]!_q}{[r_{i_0}]!_q [b-k]!_q}b \\
&=
\begin{cases}
 {d \choose m_1, \dots, m_a, d-m}  \lim_{q \rightarrow \zeta^d} \frac{1}{[b-k]_q} & \text{$s_{i_0} = 0$} \\
 {d \choose m_1, \dots, m_a, d-m}   \lim_{q \rightarrow \zeta^d} \frac{1}{[r_{i_0}]_q} & \text{$s_{i_0} = 1$}
\end{cases} \\
&= 
 {d \choose m_1, \dots, m_a, d-m}.
\end{align*}

By Corollary~\ref{symmetric-kreweras-count}, we have $X(\zeta^d) = | X^{\rot^d} |$.
\end{proof}

The following Narayana version of Theorem~\ref{kreweras-csp} 
proves a CSP involving the action of rotation on $a,b$-noncrossing partitions
with a fixed number of blocks.

\begin{theorem}
\label{narayana-csp}
Let $a < b$ be coprime, let $1 \leq k \leq a$, and let $X$ be the set of $(a,b)$-noncrossing partitions of 
$[b-1]$ with $k$ blocks.

The triple $(X, C, X(q))$ exhibits the cyclic sieving phenomenon, where $C = \ZZ_{b-1}$ acts on $X$ by rotation and
\begin{equation}
X(q) = \Nar_q(a,b,k) = \frac{1}{[a]_q} {a \brack k}_q {b-1 \brack k-1}_q
\end{equation}
is the {\em $q$-rational Narayana number}.
\end{theorem}

\begin{proof}
Reiner and Sommers proved that the $q$-Narayana numbers 
$\Nar_q(a,b,k)$ are polynomials in $q$ with nonnegative integer coefficients  \cite{ReinerSommers}.
As in the Kreweras case, no combinatorial proof of this fact is known.

Let $\zeta = e^{\frac{2 \pi i}{b-1}}$ and let $d | (b-1)$ with $1 \leq d < b-1$.  
Let $q = \frac{b-1}{d}$.
By 
Corollary~\ref{symmetric-narayana-count}, it is enough to show that 
\begin{equation*}
X(\zeta^d) = 
\begin{cases}
{d \choose \lfloor \frac{k}{d} \rfloor} {\lfloor \frac{ad}{b-1} \rfloor - 1 \choose  \lfloor \frac{k}{q} \rfloor - 1}
& \text{if $k \equiv 0$ (mod $q$)}, \\
{d \choose \lfloor \frac{k}{q} \rfloor} {\lfloor \frac{ad}{b-1} \rfloor \choose \lfloor \frac{k}{q} \rfloor} &
\text{if $k \equiv 1$ (mod $q$),} \\
0 & \text{otherwise.}
\end{cases}
\end{equation*}
The argument here is similar to that in the proof of Theorem~\ref{kreweras-csp} and is left to the 
reader.
\end{proof}

The next CSP was asked for in \cite[Subsection 6.2]{ARW}.

\begin{theorem}
\label{catalan-csp}
Let $a < b$ be coprime and let $X$ be the set of $(a,b)$-noncrossing partitions of $[b-1]$.  

The triple $(X, C, X(q))$ exhibits the cyclic sieving phenomenon, where $C = \ZZ_{b-1}$ acts on $X$ by rotation and
\begin{equation}
X(q) = \Cat_q(a,b) = \frac{1}{[a+b]_q}{a + b \brack a, b}_q
\end{equation}
is the $q$-rational Catalan number.
\end{theorem}

\begin{proof}

Let $\zeta = e^{\frac{2 \pi i}{b-1}}$ and let $d | (b-1)$ with $1 \leq d < b-1$.  
By Corollary~\ref{symmetric-catalan-count} it is enough to show that
\begin{equation*}
X(q) = { \lfloor \frac{ad}{b-1} \rfloor + d \choose d}.
\end{equation*}
The argument here is similar to that in the proof of Theorem~\ref{kreweras-csp}
and is left to the reader.
\end{proof}

Our final CSP proves 
\cite[Conjecture 5.3]{ARW}.

\begin{theorem}
\label{homogeneous-catalan-csp}
Let $(a, b)$ be coprime and let $X$ be the set of homogeneous $(a,b)$-noncrossing partitions of $[a+b-1]$.

The triple $(X, C, X(q))$ exhibits the cyclic sieving phenomenon, where $C = \ZZ_{a+b-1}$ acts on $X$ by rotation and
\begin{equation}
X(q) = \Cat_q(a,b) = \frac{1}{[a+b]_q}{a + b \brack a, b}_q
\end{equation}
is the $q$-rational Catalan number.
\end{theorem}

\begin{proof}
Let $D$ be an $a,b$-Dyck path.  Construct an $a,a+b$-Dyck path $D'$ by replacing every north
step in $D$ with a pair $NE$.   Then $D'$ has $a$ vertical runs and $\overline{\pi}(D) = \pi(D')$.  
Moreover, the map $D \mapsto D'$ gives a bijection
\begin{equation*}
\{ \text{all $a,b$-Dyck paths} \} \longrightarrow \{ \text{all $a,a+b$-Dyck paths with $a$ vertical runs} \}.
\end{equation*}
It follows that the set $\HNC(a,b)$ of homogeneous $a,b$-noncrossing partitions of $[a+b-1]$ is precisely
the set of (ordinary) $a,a+b$-noncrossing partitions of $[a+b-1]$ with precisely $a$ blocks.  
The result follows from Theorem~\ref{narayana-csp}.
\end{proof}

\section{Parking functions}
\label{Parking functions}

\subsection{Noncrossing parking functions}
Let $W$ be an irreducible real reflection group with Coxeter number $h$.
Armstrong, Reiner, and Rhoades defined a $W \times \ZZ_h$-set $\Park^{NC}_W$
called the set of {\em $W$-noncrossing parking functions} \cite{ARR}.  Given a Fuss parameter $m \geq 1$,
a Fuss extension $\Park^{NC}_W(m)$ of $\Park^{NC}_W$ was defined in \cite{RhoadesFuss}.
An increasingly strong trio of conjectures (Weak, Intermediate, and Strong) was formulated about these objects
and it was shown that the weakest of these uniformly implies various facts from $W$-Catalan theory which 
are at present only understood in a case-by-case fashion.  

In this section, we give a rational extension of the constructions in \cite{ARR, RhoadesFuss} 
when $W = \symm_a$ is the symmetric group.
This gives evidence that $\NC(a,b)$ gives the `correct' definition of rational noncrossing partitions.
Extending the work of \cite{ARR, RhoadesFuss} to other reflection groups remains an open problem.

\begin{defn}
Let $a < b$ be coprime.  An {\em $a,b$-noncrossing parking function} is a pair
$(\pi, f)$ where $\pi \in \NC(a,b)$ is an $a,b$-noncrossing partition and 
$B \mapsto f(B)$ is a labeling of the blocks of $\pi$ with subsets of $[a]$ such that
\begin{itemize}
\item we have $[a] = \biguplus_{B \in \pi} f(B)$, and
\item  for all blocks $B \in \pi$ we have $|f(B)| = \rank^{\pi}_{a,b}(B)$.
\end{itemize}
We denote by $\Park^{NC}(a,b)$ the set of all $a,b$-noncrossing parking functions.
\end{defn}

\begin{figure}
\centering
\includegraphics[scale = 0.5]{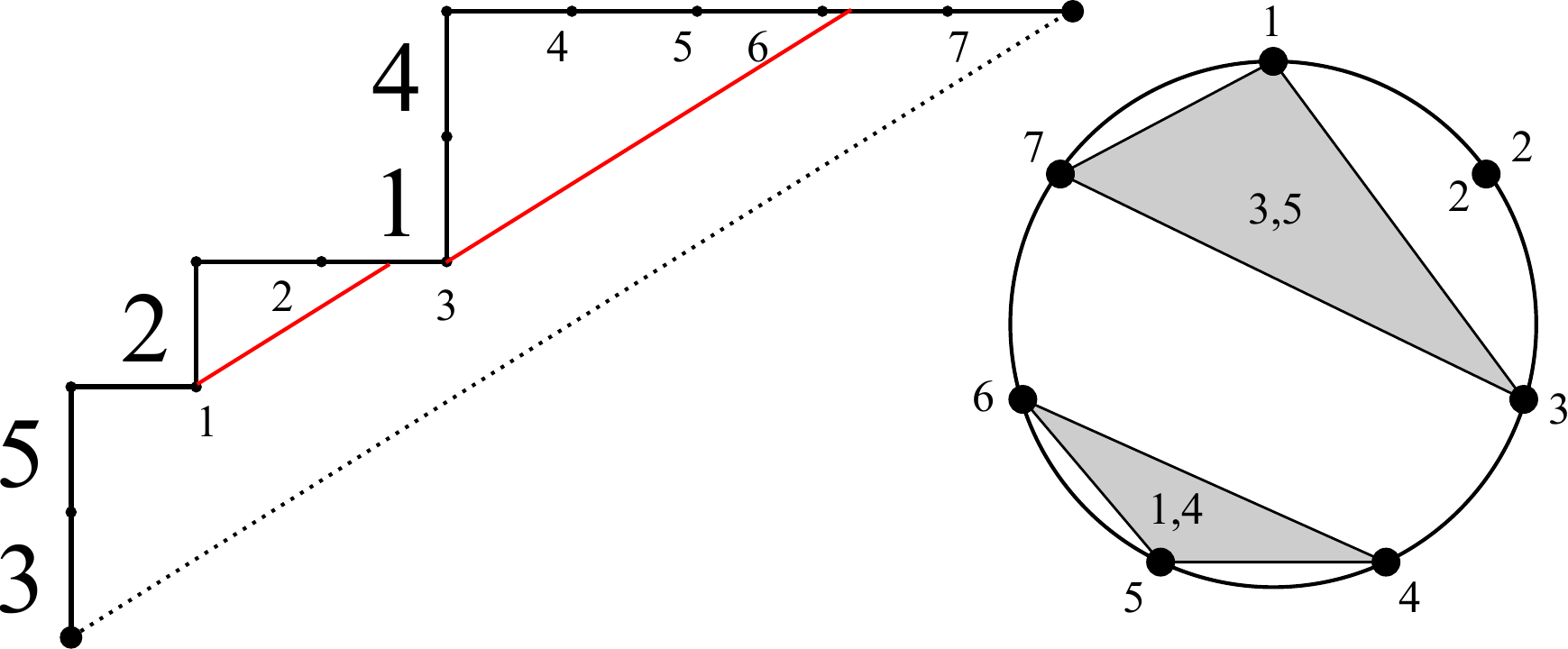}
\caption{A $5,8$-noncrossing parking function.}
\label{fig:park}
\end{figure}

An example of a $5,8$-noncrossing parking function is shown in Figure~\ref{fig:park}.
The parking function shown is $(\pi, f)$, where $\pi = \{ \{1,3,7\}, \{2\}, \{4,5,6\} \}$ and 
$f$ is the labeling 
$\{1,3,7\} \mapsto \{3,5\}, \{4,5,6\} \mapsto \{1,4\},$ and $\{2\} \mapsto \{2\}$.

By Proposition~\ref{ranks-fuss}, when $(a, b) = (n, mn+1)$, the set
$\Park^{NC}(n, mn+1)$ agrees with the construction of $\Park^{NC}_{\symm_n}(m)$ given in
\cite{RhoadesFuss}.  In the classical case $(a, b) = (n, n+1)$, the set $\Park^{NC}(n, n+1)$
appeared in the work of Edelman under the name of `$2$-noncrossing partitions' \cite{Edelman}.
The set $\Park^{NC}(a,b)$ carries an action of $\symm_a \times \ZZ_{b-1}$.

\begin{proposition}
\label{rational-action}
The set $\Park^{NC}(a,b)$ carries an action of the product group $\symm_a \times \ZZ_{b-1}$, where
$\symm_a$ acts by label permutation and $\ZZ_{b-1}$ acts by rotation.
\end{proposition}

\begin{proof}
We need to know that rotation preserves $a,b$-ranks of blocks,
which is Proposition~\ref{dihedral-invariance}.
\end{proof}

In order to state a formula for the character of the action in Proposition~\ref{rational-action}, we will need 
some notation.
Let $V = \CC^a / \langle (1, \dots, 1) \rangle$ be the reflection representation of $\symm_a$ and let
$\zeta = e^{\frac{2 \pi i}{b-1}}$.
Given 
$w \in \symm_a$ and $d \geq 0$, let $\mult_w(\zeta^d)$ be the multiplicity of $\zeta^d$ as an eigenvalue in the 
action of $w$ on $V$.  With this notation, a formula for the character $\chi$ is given by the following 
formula.

\begin{theorem}
\label{rational-weak}
Let $w \in \symm_a$ and let $g$ be a generator of $\ZZ_{b-1}$.  We have that
\begin{equation}
\label{character-formula}
\chi(w, g^d) = b^{\mult_w(\zeta^d)}
\end{equation}
for all $w \in \symm_a$ and $d \geq 0$.
\end{theorem}

The multiplicity $\mult_w(\zeta^d)$ can be read off from the cycle structure of $w$.  Namely, for 
$d | b-1$ we have that
\begin{equation}
\mult_w(\zeta^d) = \begin{cases}
\# \text{(cycles of $w$)} - 1 & \text{if $q = 1$,} \\
\# \text{(cycles of $w$ of length divisible by $q$)} & \text{otherwise,}
\end{cases}
\end{equation}
where $q = \frac{b-1}{d}$.

When $(a, b) = (n, n+1)$, 
Proposition~\ref{rational-weak} was proven in \cite[Section 8]{ARR}.
When $(a, b) = (n, mn+1)$, Proposition~\ref{rational-weak} is \cite[Proposition 8.6]{RhoadesFuss}.
The character formula of
Equation~\ref{character-formula} is a rational 
extension of the Weak Conjecture of \cite{ARR, RhoadesFuss} when $W = \symm_a$ is the symmetric
group.

The argument used to prove \cite[Proposition 8.6]{RhoadesFuss} can be combined with the enumerative
results of
Section~\ref{Modified rank sequences} to prove Theorem~\ref{rational-weak}.
We quickly illustrate how this is done.

\begin{proof} (of Theorem~\ref{rational-weak})
Let $(w, g^d) \in \symm_a \times \ZZ_{b-1}$.  We want to show that $\chi(w, g^d) = b^{\mult_w(\zeta^d)}$.
Without loss of generality, we may assume that $d | b-1$.  Let $q := \frac{b-1}{d}$.  The argument 
depends on whether $q = 1$ or $q > 1$.

{\bf Case 1:} $q = 1$.  In this case, we are ignoring the action of $\ZZ_{b-1}$ and considering 
$\Park^{NC}(a,b)$ as an $\symm_a$-set.  We construct an $\symm_a$-equivariant bijection from 
$\Park^{NC}(a,b)$ to another $\symm_a$-set which is known to have the correct character.

Let $\Park_{a,b}$ be the set of all sequences $(p_1, p_2, \dots, p_a)$ of positive integers whose nondecreasing
rearrangement $(p'_1 \leq p'_2 \leq \cdots \leq p'_a)$ satisfies 
$p'_i \leq \frac{b}{a} (i-1) + 1$.  
Equivalently, the histogram with left-to-right heights $(p'_1 - 1, p'_2 - 1, \dots, p'_a - 1)$
stays below the line $y = \frac{b}{a} x$.
 Sequences in $\Park_{a,b}$ are called {\em rational slope parking functions}.

The symmetric group $\symm_a$ acts on $\Park_{a,b}$ by
$w.(p_1, p_2, \dots, p_a) := (p_{w(1)}, p_{w(2)}, \dots, p_{w(a)})$.  It is known that the character of this action
is given by Equation~\ref{character-formula} with $\zeta = 1$.

We build an $\symm_a$-equivariant bijection $\varphi: \Park^{NC}(a,b) \xrightarrow{\sim} \Park_{a,b}$ as follows.
Let $(\pi, f)$ be an $a,b$-noncrossing parking function.  Define a sequence $(p_1, p_2, \dots, p_a)$ by letting
$p_i = \min(B)$, where $B$ is the unique block of $\pi$ satisfying $i \in f(B)$. 
 Proposition~\ref{recover-dyck-path}
guarantees that $(p_1, p_2, \dots, p_a)$ is a sequence in $\Park_{a,b}$, so that the assignment
$\varphi: (\pi, f) \mapsto (p_1, p_2, \dots, p_a)$ gives a well defined function
$\varphi: \Park^{NC}(a,b) \rightarrow \Park_{a,b}$.  It is clear that $\varphi$ is $\symm_a$-equivariant.
Moreover, if $\varphi(\pi, f) = (p_1, p_2, \dots, p_a)$, then we can recover both the minimal elements of the blocks
of $\pi$ (these are the entries appearing in $p_1, p_2, \dots, p_a$) and the ranks of these blocks
(the rank of a block $B$ is the number of times $\min(B)$ occurs in $p_1, p_2, \dots, p_a$).  
Proposition~\ref{recover-dyck-path} says that the $a,b$-noncrossing partition $\pi$ is therefore determined
from $(p_1, p_2, \dots, p_a)$.  It is easy to see
that we can determine the entire $a,b$-noncrossing function $(\pi, f)$, so that $\varphi$ is an
$\symm_a$-equivariant bijection.

{\bf Case 2:}  $q > 1$.  This argument is a rational extension of \cite[Section 8]{RhoadesFuss}.  
Let $r_q(w)$ be the number of cycles of $w$ having length divisible by $q$.  We need to show that 
\begin{equation}
\label{fixed-point-count}
|\Park^{NC}(a,b)^{(w,g^d)}| = b^{r_q(w)}, 
\end{equation}
where $\Park^{NC}(a,b)^{(w,g^d)}$ is the set of $a,b$-noncrossing parking functions fixed by $(w,g^d)$.
The idea is to show that both sides of Equation~\ref{fixed-point-count} count a certain set of functions.

Let $g$ act on the set $[b-1] \cup \{0\}$ by the permutation $(1, 2, \dots, b-1)(0)$.  A function
$e: [a] \rightarrow [b-1] \cup \{0\}$ is said to be {\em $(w,g^d)$-equivariant} if we have
\begin{equation*}
e(w(j)) = g^d e(j)
\end{equation*}
for every $1 \leq j \leq a$.  We claim that the number of $(w,g^d)$-equivariant functions
$e: [a] \rightarrow [b-1] \cup \{0\}$ is equal to $b^{r_q(w)}$.  Indeed, the values of $e$ on any cycle of $w$
are determined by the value of $e$ on any representative of that cycle.  Moreover, unless a cycle of $w$ 
has length divisible by $q$, the relation $e(w(j)) = g^d e(j)$ forces $e(j) = 0$ for all $j$ belonging to that cycle.
For every cycle of $w$ having length divisible by $q$, we have $b$ choices for $e(j)$, where $j$ is a chosen
cycle representative.

By considering set partitions coming from preimages, we can get another expression for the number
of $(w,g^d)$-admissible functions $[a] \rightarrow [b-1] \cup \{0\}$.
A set partition $\sigma = \{B_1, B_2, \dots \}$ of $[a]$ is called {\em $(w,q)$-admissible} if 
\begin{itemize}
\item  $\sigma$ is $w$-stable in the sense that 
$w(\sigma) = \{w(B_1), w(B_2), \dots \} = \sigma$,
\item at most one block $B_{i_0}$ of $\sigma$ is itself $w$-stable in the sense that $w(B_{i_0}) = B_{i_0}$, and
\item for any block $B_i$ of $\sigma$ which is not $w$-stable, the blocks 
$B_i, w(B_i), w^2(B_i), \dots , w^{q-1}(B_i)$ are pairwise distinct, and $w^q(B_i) = B_i$.
\end{itemize}
It is straightforward to see that, for any $(w,g^d)$-equivariant function $e: [a] \rightarrow [b-1] \cup \{0\}$, the set 
partition $\sigma$ of $[a]$ defined by $i \sim j$ if and only if $e(i) = e(j)$ is 
$(w,q)$-admissible.  On the other hand, the same argument as in
\cite[Proof of Lemma 8.4]{RhoadesFuss} shows that the number of $(w,g^d)$-equivariant functions
$e: [a] \rightarrow [b-1] \cup \{0\}$ which induce a fixed $(w,q)$-admissible set partition $\sigma$ of $[a]$
is  $(b-1)(b-1-q)(b-1-2q) \cdots (b-1 - (t_{\sigma} - 1)q)$, where $t_{\sigma}$ is the number of 
non-singleton $w$-orbits of blocks of $\sigma$.  Combining this with the last paragraph, we get that
\begin{equation}
\label{twelvefold}
b^{r_q(w)} = \sum_{\sigma} (b-1)(b-1-q)(b-1-2q) \cdots (b-1 - (t_{\sigma} - 1)q),
\end{equation}
where the sum is over all $(w,q)$-admissible set partitions $\sigma$ of $[a]$.

To relate Equation~\ref{twelvefold} to parking functions, for $(\pi, f) \in \Park^{NC}(a,b)$ we let 
$\tau(\pi, f)$ be the set partition of $[a]$ defined by $i \sim j$ if and only if $i, j \in f(B)$ for some block $B \in \pi$.
If $(\pi, f) \in \Park^{NC}(a,b)^{(w,g^d)}$, it follows that $\tau(\pi, f)$ is a $(w,q)$-admissible set partition
of $[a]$.  On the other hand, if $\sigma$ is a fixed $(w,q)$-admissible partition of $[a]$, we claim that the number of
parking functions $(\pi, f) \in \Park^{NC}(a,b)^{(w,g^d)}$ with $\tau(\pi, f) = \sigma$ equals 
$(b-1)(b-1-q)(b-1-2q) \cdots (b-1 - (t_{\sigma} - 1)q)$, where $t_{\sigma}$ is the number of nonsingleton $w$-orbits
of blocks in $\sigma$.

To see why this claim is true, we consider how to construct an $a,b$-noncrossing parking function
$(\pi, f) \in \Park^{NC}(a,b)^{(w,g^d)}$ with $\tau(\pi, f) = \sigma$.  
This argument is almost the same as that proving \cite[Lemma 8.5]{RhoadesFuss}, but it will rely on 
Corollary~\ref{symmetric-kreweras-count}.

We first construct an $a,b$-noncrossing partition $\pi$ which is invariant under $\rot^d$.  If 
$\sigma$ has $m_i$ non singleton $w$-orbits of blocks of size $i$ for $1 \leq i \leq a$, then $\pi$ must have
$m_i$ $\langle \rot^d \rangle$-orbits of non-central blocks of rank $i$ for $1 \leq i \leq a$.  
Corollary~\ref{symmetric-kreweras-count} says that the number of such partitions $\pi$ is the multinomial
coefficient ${d \choose m_1, m_2, \dots, m_a, d- t_{\sigma}}$.  With $\pi$ fixed,
we consider how to build the labeling $f$ of the blocks of $\pi$.  The labeling $f$ must pair off the 
$\langle \rot^d \rangle$-orbits of non-central blocks of $\pi$ of rank $i$ with the
non-singleton $w$-orbits of blocks of $\sigma$ of size $i$.  For every $i$, there are $m_i!$ ways to do this matching.
As each of these orbits has size $q$, we also have $q$ ways to rotate $f$ within each orbit after this matching
is chosen.  In summary, we have that the number of pairs $(\pi, f) \in \Park^{NC}(a,b)^{(g,w^d)}$ satisfying
$\tau(\pi, f) = \sigma$ is
\begin{align*}
q^{m_1} \cdots q^{m_a} m_1! \cdots m_a! {d \choose m_1, \dots, m_a, d- t_{\sigma}} &=
q^{m_1} \cdots q^{m_a} \frac{d!}{(d-t_{\sigma})!} \\
&= (b-1)(b-1-q)(b-1-2q) \cdots (b-1 - (t_{\sigma} - 1)q).
\end{align*}
Applying Equation~\ref{twelvefold}, we obtain Equation~\ref{fixed-point-count}, completing the proof.
\end{proof}

Theorem~\ref{rational-weak} can be strengthened to prove a rational
analog of the Generic Strong Conjecture of \cite{RhoadesEvidence} in type A.
Let $V$ be the (complexified) reflection representation of $\symm_a$, let 
$\CC[V] = \bigoplus_{d \geq 0} \CC[V]_d$ be its polynomial ring, and equip 
$\CC[V]$ with the graded action of $\symm_a \times \ZZ_{b-1}$ given by letting 
$\symm_a$ act by linear substitutions and the generator $g$ of $\ZZ_{b-1}$ scale by $(e^{\frac{2 \pi i}{b-1}})^d$
in degree $d$.  We identify $\CC[V]_1$ with the dual space $V^*$ and consider the set of 
$\CC[\symm_a]$-equivariant linear maps
$\mathrm{Hom}_{\CC[\symm_a]}(V^*, \CC[V]_b)$ as an affine complex space.
We refer the reader to \cite{RhoadesEvidence} for the definitions of the objects in the following result.

\begin{theorem}
\label{rational-generic-strong}
Let $\mathcal{R} \subset \mathrm{Hom}_{\CC[\symm_a]}(V^*, \CC[V]_b)$ be the set of 
$\Theta \in \mathrm{Hom}_{\CC[\symm_a]}(V^*, \CC[V]_b)$ such that the `parking locus'
$V^{\Theta}(b) \subset V$ cut out by the ideal
$\langle \Theta(x_1) - x_1, \dots, \Theta(x_{a-1}) - x_{a-1} \rangle \subset \CC[V]$ is reduced (here 
$x_1, \dots, x_{a-1}$ is any basis of $V^*$).  
For any $\Theta \in \mathcal{R}$, there exists an equivariant bijection of 
$\symm_a \times \ZZ_{b-1}$-sets
\begin{equation}
V^{\Theta}(b) \cong_{\symm_a \times \ZZ_{b-1}} \Park^{NC}(a,b).
\end{equation}
Moreover, there exists a nonempty Zariski open subset 
$\mathcal{U} \subseteq \mathrm{Hom}_{\CC[\symm_a]}(V^*, \CC[V]_b)$ such that 
$\mathcal{U} \subseteq \mathcal{R}$.
\end{theorem}

The proof of Theorem~\ref{rational-generic-strong} is almost a word-for-word recreation of 
\cite[Sections 4, 5]{RhoadesEvidence}.  One need only replace the reference to the proof 
of \cite[Lemma 8.5]{RhoadesFuss} in the proof of \cite[Lemma 4.6]{RhoadesEvidence} with the corresponding 
argument the fifth paragraph of Case 2 in the proof of Theorem~\ref{rational-weak}
(which ultimately relies on Corollary~\ref{symmetric-kreweras-count}).

\section{Closing Remarks}
\label{Closing Remarks}

This paper has focused entirely on rational Catalan theory for the symmetric group.
The more ambitious problem of extending rational Catalan combinatorics to other reflection groups 
$W$ is almost entirely open.  However, the results of this paper give a roadmap for defining rational
noncrossing partitions for the hyperoctohedral group.

Let $W(B_n)$ denote the hyperoctohedral group of signed permutations of $[n]$.  In the classical and Fuss-Catalan
cases, objects associated to the group $W(B_n)$ are obtained by considering those attached to 
the `doubled' symmetric group $S_{2n}$ which are invariant under antipodal symmetry.
When $(a,d) \rightarrow (2n, \frac{b-1}{2})$, the formulas in
Corollaries~\ref{symmetric-kreweras-count}, \ref{symmetric-narayana-count}, and 
\ref{symmetric-catalan-count} reduce to the hyperoctohedral analogs of the rational
Kreweras, Narayana, and Catalan numbers (here we view $2n$ as the Coxeter number of 
$W(B_n)$ and let the rational parameter $b$ be coprime to $2n$).
Thus, restricting to objects with antipodal symmetry gives the correct numerology for type B, even
in the rational setting.
It would be interesting to see how far the techniques of this paper can be extended to develop
on rational Catalan combinatorics outside of type A.

\section{Acknowledgements}
\label{Acknowledgements}

The authors are grateful to Drew Armstrong and Vic Reiner for helpful conversations.
B. Rhoades was partially supported by NSF grants
DMS - 1068861 and DMS - 1500838.

\end{document}